\documentclass[11pt, leqno]{amsart}
\usepackage{amsmath, amssymb, latexsym}
\usepackage{mathrsfs}
\usepackage{enumerate}
\usepackage{color}
\usepackage[colorlinks=true, pdfstartview=FitV, linkcolor=blue,%
citecolor=blue, urlcolor=blue]{hyperref}
\usepackage[all, knot]{xy}
\xyoption{arc}
\usepackage{tikz-cd}
\usetikzlibrary{cd}
\usepackage{tikz}
\usepackage{amsmath}
\usetikzlibrary{cd}
\usetikzlibrary{arrows.meta}

\numberwithin{equation}{section}

\newtheorem{theorem}{Theorem}[section]
\newtheorem{corollary}[theorem]{Corollary}
\newtheorem{lemma}[theorem]{Lemma}
\newtheorem{proposition}[theorem]{Proposition}

\newtheorem{example}[theorem]{Example}

\theoremstyle{definition}
\newtheorem{definition}[theorem]{Definition}
\newtheorem{remark}[theorem]{Remark}

\newcommand{\Hom}{\operatorname{Hom}}

\newcommand{\Q}{\mathbf{Q}}

\newcommand{\Z}{\mathbf{Z}}
\newcommand{\h}{\mathfrak{h}}
\newcommand{\g}{\mathfrak{g}}

\title[Quantum Borcherds-Bozec algebras]
{Quantum Borcherds-Bozec algebras and\\ their integrable representations}

\author[Seok-Jin Kang]{Seok-Jin Kang$^{*}$}
\address{Korea Research Institute of Arts and Mathematics, 
Asan-si, Chungcheongnam-do, 31551, Korea}
\email{soccerkang@hotmail.com}

\thanks{$^{*}$ This research was supported by Hankuk University of Foreign Studies Research Fund.}

\author[Young Rock Kim]{Young Rock Kim${}^{**}$}
\address{Graduate School of Education, Hankuk University of Foreign Studies, Seoul, 02450,  Korea}
\email{rocky777@hufs.ac.kr} %
\thanks{${}^{**}$ This research was supported by the Basic Science Research Program of the NRF (Korea) under grant No. 2015R1D1A1A01059643.}

\keywords{quantum Borcherds-Bozec algebra, triangular decomposition, integrable representation, character formula,  
complete reducibility}

\subjclass[2010] {17B37, 17B67, 16G20}
\begin{document}

\begin{abstract}
We investigate the fundamental properties of quantum Borcherds-Bozec algebras and their representations. Among others, we prove that the the quantum Borcherds-Bozec algebras have a triangular decomposition and the category of integrable representations is semi-simple. 
\end{abstract}

\maketitle

\section*{Introduction}

\vskip 2mm

The {\it quantum Borcherds-Bozec algebras} were introduced by T. Bozec in his geometric investigation of the representation theory of quivers with loops \cite{Bozec2014b, Bozec2014c}. He gave a construction of Lusztig's canonical basis for the positive half of a quantum Borcherds-Bozec algebra in terms of simple perverse sheaves on the representation variety of quivers with loops (cf. \cite{Lus90}). 

\vskip 2mm

On algebraic side, he developed the essential part of crystal basis theory for quantum Borcherds-Bozec algebras. First of all, he defined the Kashiwara operators on the integrable representations and on the negative half of a quantum Borcherds-Bozec algebra, which provides an important framework for Kashiwara's grand-loop argument (cf. \cite{Kas91}). 

\vskip 2mm 

He went on to define the notion of abstract crystals for quantum Borcherds-Bozec algebras and gave a geometric construction of the crystal for the negative half of a quantum Borcherds-Bozec algebra based on the theory of Lusztig's quiver varieties (cf. \cite{KS1997, KKS2009}). Moreover, using Nakajima's quiver varieties, he also gave a geometric construction of crystals for the integrable highest weight representations (cf. \cite{Saito2002, KKS2012}).   

\vskip 2mm

The purpose of this paper is to provide a rigorous foundation for the theory of quantum Borcherds-Bozec algebras and their representations. Among others, we prove that the quantum Borcherds-Bozec algebras have a triangular decomposition and the category of integrable representations is semi-simple\,; i.e., all the integrable representations are completely reducible. 

\vskip 2mm  

Compared with the theory of quantum Kac-Moody algebras and quantum Borcherds algebras, one of the main difficulties lies in the fact that the commutation relations between positive part and negative part are much more complicated for quantum Borcherds-Bozec algebras due to the Drinfeld-type defining relations. Also the co-multiplication formulas need a lot more careful treatment when we show that the quantum Borcherds-Bozec algebras have a triangular decomposition (Theorem \ref{thm:triangular}). In fact, we first need to verify that we have a well-defined co-multiplication on the quantum Borcherds-Bozec algebras (Proposition \ref{prop:co-mult}).  By a detailed analysis of Drinfeld-type commutation relations, we prove  one of the key ingredients for our main results (Proposition \ref{prop:hw}), which will lead to a characterization of irreducible highest weight representations with dominant integral highest weights.  Thanks to the character formula for integrable highest weight representations \cite{BSV2016}, we can follow the outline given in \cite{HK02, JKK05} to prove that all the integrable representations are completely reducible (Theorem \ref{thm:semi-simple}).   
 
\vskip 2mm 

This paper is organized as follows. In Section1, we recall Bozec's construction of quantum Borcherds-Bozec algebras.  Section 2 is devoted to a detailed analysis of Drinfeld-type commutation relations. We investigate the structure of quantum string algebras and prove that there exists a well-defined co-multiplication on the quantum Borcherds-Bozec algebras. In Section 3, we show that the quantum Borcherds-Bozec algebras have a triangular decomposition.  In Section 4, using the detailed analysis of Drinfeld-type commutation relations, we prove Proposition \ref{prop:hw}, a key ingredient for our main results. Finally, in Section 5, we prove that all the integrable representations are completely reducible. 

\vskip 2mm

{\it Acknowledgements}. The first author would like to express his sincere gratitude to Harbin Engineering University for their hospitality during his visit in July and November, 2019.

\vskip 8mm

\section{Quantum Borcherds-Bozec algebras}

\vskip 2mm

We first review Bozec's construction of quantum Borcherds-Bozec algebras \cite{Bozec2014c}.  

\vskip 2mm 

Let $I$ be an index set which can be countably infinite. An integer-valued matrix $A=(a_{ij})_{i,j \in I}$ is called an {\it
even symmetrizable Borcherds-Cartan matrix} if it satisfies the following conditions:
\begin{itemize}
\item[(i)] $a_{ii}=2, 0, -2, -4, ...$,

\item[(ii)] $a_{ij}\le 0$ for $i \neq j$,

\item[(iii)] there exists a diagonal matrix $D=\text{diag} (s_{i} \in \Z_{>0} \mid i \in I)$ such that $DA$ is symmetric.
\end{itemize}

\vskip 2mm

\noindent Set $I^{\text{re}}=\{i \in I \mid a_{ii}=2 \}$,
$I^{\text{im}}=\{i \in I \mid a_{ii} \le 0\}$ and
$I^{\text{iso}}=\{i \in I \mid a_{ii}=0 \}$.

\vskip 3mm

A {\it Borcherds-Cartan datum} consists of :

\ \ (a) an even symmetrizable Borcherds-Cartan matrix $A=(a_{ij})_{i,j \in I}$,

\ \ (b) a free abelian group $P$, the {\it weight lattice},

\ \ (c) $\Pi=\{\alpha_{i} \in P  \mid i \in I \}$, the set of {\it simple roots},

\ \ (d) $P^{\vee} := \Hom(P, \Z)$, the {\it dual weight lattice},

\ \ (e) $\Pi^{\vee}=\{h_i \in P^{\vee} \mid i \in I \}$, the set of {\it simple coroots}

\vskip 2mm

\noindent satisfying the following conditions

\vskip 1mm

\begin{itemize}

\item[(i)] $\langle h_i, \alpha_j \rangle = a_{ij}$ for all $i, j \in I$,

\item[(ii)] $\Pi$ is linearly independent over $\Q$,

\item[(iii)] for each $i \in I$, there exists an element $\Lambda_{i} \in P$ such that $$\langle h_j , \Lambda_i
\rangle = \delta_{ij} \ \ \text{for all} \ i, j \in I.$$
\end{itemize}

\vskip 2mm

Given an even symmetrizable Borcherds-Cartan matrix, it can be shown that such a Borcherds-Cartan datum always exists, which 
is not necessarily unique. The $\Lambda_i$  $(i \in I)$ are called the {\it fundamental weights}.

\vskip 3mm

We denote by
$$P^{+}:=\{\lambda \in P \mid \langle h_i, \lambda \rangle \ge 0 \\text{for all} \ i \in I \}$$ 
the set of {\it dominant integral weights}. The free abelian group $Q:= \bigoplus_{i \in I} \Z \, \alpha_i$ is called the {\it root lattice}. 
Set $Q_{+}: = \sum_{i \in I} \Z_{\ge 0}\, \alpha_{i}$ and $Q_{-}: = -Q_{+}$.  For $\beta = \sum k_i \alpha_i \in Q\textbf{}_{+}$, we define its {\it height} to be $|\beta|:=\sum k_i$.

\vskip 3mm

Let ${\mathfrak h} := \Q \otimes_{\Z} P^{\vee}$ be the {\it Cartan subalgebra}.  We define 
a partial ordering on ${\mathfrak h}^{*}$ by setting $\lambda \ge \mu$  if and only if $\lambda - \mu  \in Q_{+}$ for 
$\lambda, \mu \in {\mathfrak h}^{*}$.  

\vskip 3mm 

Since $A$ is symmetrizable and $\Pi$  is linearly independent over $\Q$, 
there exists a non-degenerate symmetric bilinear
form $( \ , \ )$ on ${\mathfrak h}^{*}$ satisfying
$$(\alpha_{i}, \lambda) = s_{i} \langle h_{i}, \lambda \rangle \quad
\text{for all}  \ \lambda \in {\mathfrak h}^{*}.$$

\vskip 3mm

For each  $i \in I^{\text{re}}$, we deinfe the {\it simple reflection} $r_{i} \in {\mathfrak h}^{*}$ by 
$$r_{i}(\lambda)= \lambda - \langle h_{i}, \lambda \rangle \,  \alpha_{i} \ \ \text{for} \ \lambda \in {\mathfrak h}^{*}.$$
The subgroup $W$ of $GL({\mathfrak h}^{*})$ generated by the simple reflections $r_{i}$ $(i \in I^{\text{re}})$ is called the {\it Weyl group} of the Borcherds-Cartan datum given above.  It is easy to check that $(\ , \ )$ is $W$-invariant.

\vskip 2mm 

We now proceed to define the notion of {\it quantum Borcherds-Bozec algebras}. Let $q$ be an indeterminate and set 
$$q_{i}  = q^{s_i}, \quad  q_{(i)} = q^{\frac{(\alpha_{i}, \alpha_{i})}{2}} = q_{i}^{\frac{a_{ii}}{2}}.$$ 
For each $i \in I^{\text{re}}$ and $n \in \Z_{> 0}$, we define  
$$[n]_{i} = \dfrac{q_{i}^{} - q^{-n}} {q_{i} - q_{i}^{-1}}, \quad [n]_{i} ! = [n]_{i} [n-1]_{i} \cdots [1]_{i}, \quad \left[  \begin{matrix} n \\ k\end{matrix}\right]_{i} = \dfrac{[n]_{i}!} {[k]_{i}! [n-k]_{i}!}.$$

Set $I^{\infty}:= (I^{\text{re}} \times \{1\}) \cup (I^{\text{im}}
\times \Z_{>0})$. For simplicity, we will often write $i$ for $(i,1)$ $(i \in I^{\text{re}})$.

\vskip 2mm

Let ${\mathscr F} = \Q(q) \langle f_{il} \mid (i,l) \in I^{\infty} \rangle$ 
be the free associative algebra defined on the set of alphabet 
$\{f_{il} \mid (i,l) \in I^{\infty} \}$. By setting $\text{deg} \, f_{il} = - l \alpha_{i}$, ${\mathscr F}$ becomes a
$Q_{-}$-graded algebra. 

\vskip 2mm 

We define a {\it twisted} multiplication on ${\mathscr F} \otimes {\mathscr F}$ by
\begin{equation} \label{eq:twisted}
(a_{1} \otimes a_{2})  (b_{1} \otimes b_{2}) = q^{-(\text{deg}\, a_{2}, \text{deg}\, b_{1})} a_{1} b_{1} \otimes a_{2} b_{2}  \quad \text{for} \  a_{1}, a_{2}, b_{1}, b_{2} \in {\mathscr F}.
\end{equation}

It can be shown that  there is an algebra homomorphism (called the {\it co-multiplication}) $\delta: {\mathscr F} \rightarrow {\mathscr F} \otimes {\mathscr F}$ given by 
\begin{equation} \label{eq:comult}
\delta(f_{il}) = \sum_{m+n=l} q_{(i)}^{-mn} f_{im} \otimes f_{in}.
\end{equation}
Here, we understand $f_{i0}=1$, $f_{il}=0$ for $l < 0$. 

\vskip 3mm

\begin{proposition}\cite{Lus10}\label{L-bilinear}
{\rm For each $\tau = (\tau_{il})_{(i,l)\in I^{\infty}}$ with $\tau_{il} \in \Q(q)$, there exists a symmetric bilinear form 
$(\ , \ )_{L}$ on ${\mathscr F}$ satisfying the following conditions:

\begin{enumerate}

\item[(a)] $(x, y)_{L} =0$  \ unless \ $\text{deg}\, x = \text{deg}\, y$,

\item[(b)] $(f_{il}, f_{il})_{L} = \tau_{il}$ \  for all $(i,l) \in I^{\infty}$,

\item[(c)] $(x, yz)_{L} = (\delta(x), y \otimes z)_{L} \ \ \text{for all} \ x, y, z \in {\mathscr F}$.
\end{enumerate}}
\end{proposition}

\vskip 3mm

From now on, we assume that 
\begin{equation} \label{eq:assumption}
\tau_{il} \in 1 + q \Z_{\ge 0} [[q]] \ \ \text{for all} \ (i,l) \in I^{\infty}.
\end{equation}

\vskip 3mm 

We define $\widehat{U}$ to be the $\Q(q)$-algebra with ${\mathbf 1}$ generated by the elements $q^{h}$ $(h \in P^{\vee})$ and $e_{il}$, $f_{il}$ $((i,l)\in I^{\infty})$ with the defining relations

\begin{equation}\label{eq:Uhat}
\begin{aligned}
& q^{0} = {\mathbf 1}, \ \ q^{h} \, q^{h'} = q^{h + h'} \ \ \text{for} \ h, h' \in P^{\vee}, \\
& q^{h} \, e_{jl} \, q^{-h} = q^{l \langle h, \alpha_{j}\rangle} e_{jl}, \ \ q^{h} \, f_{jl} \, q^{-h} = q^{-l \langle h, \alpha_{j}\rangle} f_{jl} \ \ \text{for} \ h \in P^{\vee}, (j,l) \in I^{\infty}, \\
& \sum_{k=0}^{1- l a_{ij}} (-1)^{k} \left[ \begin{matrix} 1- l a_{ij} \\ k\end{matrix} \right]_{i} \, e_{i}^{1- l a_{ij} -k} e_{jl} \, e_{i}^{k} =0,  \\
& \sum_{k=0}^{1- l a_{ij}} (-1)^{k} \left[ \begin{matrix} 1- l a_{ij} \\ k\end{matrix} \right]_{i} \, f_{i}^{1- l a_{ij} -k} f_{jl} \, f_{i}^{k} =0 \ \ \text{for} \ i \in I^{\text{re}}, \, i \neq (j,l) \in I^{\infty},\\
& \ \  e_{ik} \, e_{jl} - e_{jl} \, e_{ik} = 0, \ \,  f_{ik} \,f_{jl} - f_{jl} \,f_{ik} =0 \ \ \ \text{for} \, \ a_{ij}=0.
\end{aligned}
\end{equation}

In \cite{Bozec2014b}, Bozec showed that one can define an algebra homomorphism called the ({\it co-multiplication}) $\Delta: {\widehat U} \rightarrow {\widehat U} \otimes {\widehat U}$ given by 
\begin{equation} \label{eq:co-mult}
\begin{aligned}
& \Delta(q^{h}) = q^{h} \otimes q^{h}, \\
& \Delta(e_{il}) = \sum_{m+n=l} q_{(i)}^{mn} e_{im} \otimes K_{i}^{-m} e_{in}, \\
&  \Delta(f_{il}) = \sum_{m+n=l} q_{(i)}^{-mn} f_{im} K_{i}^{n} \otimes f_{in},
\end{aligned}
\end{equation}
where $K_{i}= q^{s_i h_i}$ $(i \in I)$.  
Furthermore, Bozec also showed that one can extend $(\ , \ )_{L}$ to a symmetric bilinaer form $(\ , \ )_{L}$ on $\widehat{U}$ satisfying 
\begin{equation}
\begin{aligned}
& (q^{h}, K_{j})_{L} = q^{-\langle h, \alpha_{j} \rangle}, \\
& (q^{h}, e_{il})_{L} = (q^{h}, f_{il})_{L} =0, \\
& (e_{ik}, e_{jl})_{L} = (f_{ik}, f_{jl})_{L} = \delta_{ij} \delta_{kl} \tau_{ik}.
\end{aligned}
\end{equation}

\vskip 2mm 

Define an involution $\omega : \widehat{U} \rightarrow \widehat{U}$ by 
\begin{equation} \label{eq:involution}
\omega(q^h) =q^{-h}, \ \ \omega(e_{il}) = f_{il}, \ \ \omega(f_{il}) = e_{il} \ \ \text{for} \ h \in P^{\vee}, \ (i,l)\in I^{\infty}. 
\end{equation}

\vskip 2mm 

For $x \in \widehat{U}$, following the Sweedler's notation \cite{Sweedler}, write
\begin{equation} \label{eq:Sweedler}
\Delta(x) = \sum x_{(1)} \otimes x_{(2)}. 
\end{equation}

\vskip 3mm 

\begin{definition} \label{def:qBB} \ 
{\rm Given a Borcherds-Cartan datum, the {\it quantum Borcherds-Bozec algebra} $U_{q}(\g)$  is defined to be the quotient algebra of $\widehat{U}$ by the defining relations 

\begin{equation} \label{eq:qBB}
\sum (a_{(1)}, b_{(2)})_{L} \, \omega(b_{(1)}) a_{(2)} = \sum (a_{(2)}, b_{(1)})_{L} \, a_{(1)} \omega(b_{(2)}) \ \ \text{for all} \ a, b \in \widehat{U}.
\end{equation}
}
\end{definition}

\vskip 8mm

\section{Quantum string algebras}

\vskip 2mm 

Recall that for all $(i,k), (j,l) \in I^{\infty}$, we have the co-multiplication formulas on $\widehat{U}$\,:
\begin{equation*}
\Delta(f_{ik}) = \sum_{m+n =k} q_{(i)}^{-mn} f_{im} K_{i}^{n} \otimes f_{in}, \ \ 
\Delta(f_{jl}) = \sum_{r+s=l} q_{(j)}^{-rs} f_{jr} K_{j}^{s} \otimes f_{js}.
\end{equation*}
Then the defining relation \eqref{eq:qBB} yields 
\begin{equation} \label{eq:qBB-1}
\begin{aligned}
& \sum_{\substack{m+n=k \\ r+s=l}} q_{(i)}^{-mn} q_{(j)}^{-rs} \, (f_{im} K_{i}^{n}, f_{js})_{L} \, e_{jr}K_{j}^{-s} f_{in} \\
& \qquad = \sum_{\substack{m+n=k \\ r+s=l}} q_{(i)}^{-mn} q_{(j)}^{-rs} \, (f_{in},  f_{jr} K_{j}^{s})_{L} \, f_{im} K_{i}^{n} e_{js}.
\end{aligned}
\end{equation}

\vskip 3mm 

Suppose  $i \neq j$. Then we have $(f_{im}K_{i}^{n}, f_{js})_{L} =0$ unless $m=0$, $s=0$, in which case, $n=k$, $r=l$. Hence the left-hand side of \eqref{eq:qBB-1} is equal to $e_{jl} f_{ik}$. Similarly, $(f_{in}, f_{jr}K_{j}^{s})_{L}=0$ implies $n=r=0$, $m=k$, $s=l$, and the right-hand side of \eqref{eq:qBB-1} is the same as $f_{ik}e_{jl}$. Hence we obtain
\begin{equation} \label{eq:commute}
e_{jl} f_{ik} = f_{ik} e_{jl} \quad \text{for all} \ \ i \neq j, \  k, l >0.
\end{equation}

\vskip 2mm 

Now we will deal with the case when $i = j$. For each $i \in I$, we define the {\it quantum $i$-string algebra} $U_{(i)}$ to be the subalgebra of $U_{q}(\g)$ generated by $e_{il}$, $f_{ik}$, $K_{i}^{\pm 1}$ $(k,l >0)$.  We denote by $U_{(i)}^{+}$ (resp. $U_{(i)}^{-}$) the subalgebra generated by $e_{il}$ (resp. $f_{il}$) for $(i,l) \in I^{\infty}$.  

\vskip 3mm 

We return to the relation \eqref{eq:qBB-1}. Since $i=j$, we have 
\begin{equation} \label{eq:qBB-2}
\begin{aligned}
& \sum_{\substack{m+n=k \\ r+s=l}} q_{(i)}^{-mn-rs} \, (f_{im} K_{i}^{n}, f_{is})_{L} \, e_{ir}K_{i}^{-s} f_{in} \\
& \qquad = \sum_{\substack{m+n=k \\ r+s=l}} q_{(i)}^{-mn-rs} \, (f_{in},  f_{ir} K_{i}^{s})_{L} \, f_{im} K_{i}^{n} e_{is}.
\end{aligned}
\end{equation}

\vskip 2mm 

Let us denote by $L$ and $R$ the left-hand side and right-hand side of \eqref{eq:qBB-2}, respectively. Since $(f_{im} K_{i}^{n}, f_{is})_{L}=0$ unless $m=s$, we have 
\begin{equation*}
L=\sum_{\substack{m+n=k \\ r+m=l}} q_{(i)}^{-m(n+r)} \, (f_{im} K_{i}^{n}, f_{im})_{L} \, e_{ir}K_{i}^{-m} f_{in}
=\sum_{\substack{m+n=k \\ r+m=l}} q_{(i)}^{-m(n+r)} \, \tau_{im} \, e_{ir}K_{i}^{-m} f_{in}.
\end{equation*}

\vskip 1mm

\noindent
Note that 
$$K_{i}^{-m} f_{in} = q_{i}^{mn\,a_{ii}}f_{in} K_{i}^{-m} = q_{(i)}^{2mn} f_{in} K_{i}^{-m}.$$

\vskip 1mm

\noindent
Hence we get
$$L = \sum_{\substack{m+n=k \\ r+m=l}} q_{(i)}^{m(n-r)} \tau_{im} e_{ir} f_{in} K_{i}^{-m}.$$

\vskip 2mm 

On the other hand, since $(f_{in}, f_{ir} K_{i}^{s})_{L} =0$ unless $n=r$ and $K_{i}^{n} e_{is} = q_{(i)}^{2ns} e_{is} K_{i}^{n}$, we have 
$$R = \sum_{\substack{m+n=k \\ n+s=l}} q_{(i)}^{n(s-m)} \tau_{in}  f_{im} e_{is} K_{i}^{n}.$$

\vskip 2mm 

By rearraning the indices in $L$, we obtain 

\vskip 3mm

\begin{lemma} \label{lem:string}  
{\rm For all $k, l >0$, we have the following relations in $U_{(i)}$:
\begin{equation} \label{eq:string}
\sum_{\substack{m+n=k \\ n+s=l}} q_{(i)}^{n(m-s)} \tau_{in} \, e_{is} \, f_{im} \, K_{i}^{-n}
= \sum_{\substack{m+n=k \\ n+s=l}} q_{(i)}^{-n(m-s)} \tau_{in} \, f_{im}\, e_{is} \, K_{i}^{n}. 
\end{equation}
}
\end{lemma}

\vskip 3mm

Let us analyze the implication of Lemma \ref{lem:string} in more detail. 

\vskip 3mm 

\begin{example} \label{ex:real} \hfill

\vskip 2mm 
{\rm
Suppose $a_{ii}=2$. 
In this case, $k=l=1$ and we have the following two cases:

\vskip 2mm  

\ \ (i) $m=0$, $n=1$, $s=0$

\ \ (ii) $m=1$, $n=0$, $s=1$.   

\vskip 2mm 

Therefore \eqref{eq:string} implies
$$\tau_{i,1} K_{i}^{-1} + \tau_{i,0} e_{i} f_{i} = \tau_{i,1} K_{i} + \tau_{i,0} f_{i} e_{i}.$$

\vskip 1mm 

\noindent
Take $\tau_{i,0}=1$, $\tau_{i,1} = \dfrac{1}{1-q_{i}^{2}}$ and replace $f_{i}$ by $F_{i} = - q_{i} f_{i}$. Then we obtain 
 $$e_{i} F_{i} - F_{i} e_{i} = \dfrac{K_{i} - K_{i}^{-1}}{q_{i} - q_{i}^{-1}}.$$
 Hence $U_{(i)} \cong U_{q}(sl_{2})$ and the set ${\mathbf B}=\{f_{i}^{n} \mid n\ge 0 \}$ is a basis of $U_{(i)}^{-}$. 
}
\end{example}

\vskip 3mm 

\begin{example} \label{ex:iso} \hfill

\vskip 1mm

{\rm Suppose $a_{ii}=0$. In this case, by the defining relations of quantum Borcherds-Bozec algebras, we have
\begin{equation*}
e_{ik} \, e_{il} = e_{il} \, e_{ik}, \ \ f_{ik}\, f_{il} = f_{il} \, f_{ik}, \ \
 K_{i}^{\pm 1}  \, e_{il} = e_{il} \, K_{i}^{\pm 1},  \ \ K_{i}^{\pm 1} \, f_{il} = f_{il} \, K_{i}^{\pm 1}. 
\end{equation*}
By Lemma \ref{lem:string}, the algebra $U_{(i)}$ has the additional relation
\begin{equation*} 
\sum_{\substack{m+n=k \\ n+s=l}} \tau_{in} \, e_{is} \, f_{im} \, K_{i}^{-n}
= \sum_{\substack{m+n=k \\ n+s=l}}  \tau_{in} \, f_{im}\, e_{is} \, K_{i}^{n}. 
\end{equation*}

\vskip 1mm

\noindent
We will call $U_{(i)}$ the {\it quantum twisted Heisenberg algebra}. 

\vskip 2mm 

For each $l>0$, let ${\mathbf c} = (c_{1},  c_{2},  \ldots, c_{l})$ be a {\it partition} of $l$ and define 
$f_{i, {\mathbf c}} := f_{i, c_{1}} f_{i, c_{2}} \cdots f_{i, c_{l}}$. Set ${\mathbf B}_{l}= \{ f_{i, {\mathbf c}} \mid \text{${\mathbf c}$ is a partition of $l$} \}$. Then ${\mathbf B}:= \bigcup_{l\ge 0} {\mathbf B}_{l}$  is a basis of $U_{(i)}$, where ${\mathbf B}_{0} = \{1\}$.  
}
\end{example}

\vskip 3mm 

\begin{example} \label{ex:non-iso} \hfill

\vskip 1mm 

{\rm Suppose $a_{ii} <0$. In this case, there are no relations other than \eqref{eq:string}. In particular, $U_{(i)}^{+}$ (resp. $U_{(i)}^{-})$ is the free associative algebra generated by $e_{ik}$ (resp. $f_{ik}$) for $k>0$. 

\vskip 2mm 

For each $l>0$, let ${\mathbf c} = (c_{1},  c_{2},  \ldots, c_{l})$ be a {\it composition} of $l$ and define 
$f_{i, {\mathbf c}} := f_{i, c_{1}} f_{i, c_{2}} \cdots f_{i, c_{l}}$. Set ${\mathbf B}_{l}= \{ f_{i, {\mathbf c}} \mid \text{${\mathbf c}$ is a composition of $l$} \}$. Then ${\mathbf B}:= \bigcup_{l\ge 0} {\mathbf B}_{l}$  is a basis of $U_{(i)}^{-}$, where ${\mathbf B}_{0} = \{1\}$.  
}
\end{example}

\vskip 3mm

\vskip 2mm

For simplicity, we will often write $U$ for the quantum Borcherds-Bozec algebra $U_{q}(\g)$. 
We now prove that  $\Delta$ passes down to the co-multiplication on $U$. 

\vskip 3mm 

\begin{proposition} \label{prop:co-mult}
{\rm 
The co-multiplication $\Delta$ on $\widehat{U}$ defines an algebra homomorphism
\begin{equation}
\Delta: U \longrightarrow U \otimes U.
\end{equation}
}
\end{proposition}

\begin{proof} \, We need to prove $\Delta$ preserves the defining relations of $U$. 

\vskip 2mm 

When $i \neq j$, we have already seen that 
$$f_{ik}\, e_{jl} = e_{jl} \, f_{ik} \quad \text{for all} \ \ k, l>0.$$
Recall that 
$$\Delta(f_{ik}) = \sum_{m+n=k} q_{(i)}^{-mn} \, f_{im} K_{i}^n \otimes f_{in}, \quad 
\Delta(e_{jl}) = \sum_{r+s=l} q_{(j)}^{rs} \, e_{jr} \otimes K_{j}^{-s} e_{js}.$$
Hence
\begin{equation*}
\begin{aligned}
\Delta(f_{ik})\, \Delta(e_{jl}) & = \sum_{\substack{m+n=k \\ n+s=l}} (q_{(i)}^{-mn} \, f_{im} K_{i}^{n} \otimes f_{in} ) \, (q_{(j)}^{rs} e_{jr} \otimes K_{j}^{-r} e_{js})  \\
&= \sum_{\substack{m+n=k \\ n+s=l}} q_{(i)} ^{-mn} q_{(j)}^{rs}\,  (f_{im} K_{i}^{n} e_{jr} \otimes f_{in} K_{j}^{-r} e_{js}),
\end{aligned}
\end{equation*}
which coincides with the expression in the triangular decomposition of $U$. 

\vskip 2mm

On the other hand,
\begin{equation*}
\begin{aligned}
\Delta(e_{jl})\, \Delta(f_{ik}) & = \sum_{\substack{m+n=k \\ n+s=l}} (q_{(j)}^{rs}\,  e_{jr}  \otimes K_{j}^{-r} e_{js}) \, (q_{(i)}^{-mn} f_{im} K_{i}^{n} \otimes f_{in})  \\
&= \sum_{\substack{m+n=k \\ n+s=l}} q_{(i)} ^{-mn} q_{(j)}^{rs}\, (e_{jr} f_{im} K_{i}^{n}  \otimes K_{j}^{-r} e_{js} f_{in})\\
&=\sum_{\substack{m+n=k \\ n+s=l}} q_{(i)} ^{-mn} q_{(j)}^{rs}\, (f_{im}e_{jr} K_{i}^{n}  \otimes K_{j}^{-r} f_{in} e_{js}).
\end{aligned}
\end{equation*}

We need to change the order of the expression $e_{jr} K_{i}^{n}$ and $K_{j}^{-r} f_{in}$. Note that 
$$e_{jr} K_{i}^{n} = q_{i}^{-rna_{_ij}} K_{i}^{n} e_{jr}, \quad K_{i}^{-r} f_{in} = q^{rna_{ji}} f_{in} K_{j}^{-r}.$$
Since $A$ is symmetrizable, we have $q_{i}^{a_{ij}}=q^{s_{i} a_{ij}} = q^{s_{j} a_{ji}} =q_{j}^{a_{ji}}$, which implies 
\begin{equation*}
\begin{aligned}
\Delta(e_{jl})\, \Delta(f_{ik}) &=\sum_{\substack{m+n=k \\ n+s=l}} q_{(i)} ^{-mn} q_{(j)}^{rs} 
q_{i}^{-rna_{ij}} q_{j}^{rna_{ji}} \, (f_{im} K_{i}^{n} e_{jr}   \otimes f_{in} K_{j}^{-r} e_{js}) \\
& = \sum_{\substack{m+n=k \\ n+s=l}} 
q_{(i)}^{-mn} q_{(j)}^{rs} \, (f_{im} K_{i}^{n} e_{jr}  \otimes f_{in} K_{j}^{-r} e_{js} ).
\end{aligned}
\end{equation*}
Hence we have 
$$\Delta(f_{ik}) \Delta(e_{jl}) = \Delta(e_{jl}) \Delta(f_{ik}),$$
as desired. 

\vskip 3mm 

We will move to the case when $i=j$. 

\vskip 3mm 

{\bf Case 1\,:} Supose $k=l$. 

\vskip 2mm 

In this case, by Lemma \ref{lem:string}, for all $k>0$, we have 
\begin{equation}\label{eq:delta}
\sum_{n=0}^{k} \tau_{in} (e_{i, k-n}\, f_{i, k-n} \, K_{i}^{-n}  - f_{i, k-n}\, e_{i, k-n}\, K_{i}^{n})=0. 
\end{equation}
We would like to prove 
\begin{equation} \label{eq:delta1}
\Delta\left(\sum_{n=0}^{k} \tau_{in} (e_{i, k-n}\, f_{i, k-n} \, K_{i}^{-n}  - f_{i, k-n}\, e_{i, k-n}\, K_{i}^{n}) \right)=0. 
\end{equation}

\vskip 2mm 

We will use induction on $k$. 

\vskip 2mm

If $k=l=1$, we have 
$$\tau_{i0}(e_{i1}f_{i1} -f_{i1}e_{i1}) + \tau_{i1}(K_{i}^{-1} - K_{i})=0.$$

Hence \eqref{eq:co-mult} gives
\begin{equation*}
\begin{aligned}
& \Delta (e_{i1}) \Delta(f_{i1})=(e_{i1} \otimes K_{i}^{-1} + 1 \otimes e_{i1}) \, (f_{i1} \otimes 1 + K_{i} \otimes f_{i1}) \\
& \ \ =e_{i1}f_{i1} \otimes K_{i}^{-1} + q_{(i)}^{2} \, e_{i1} K_{i} \otimes f_{i1} K_{i}^{-1} + f_{i1} \otimes e_{i1} + K_{i} \otimes e_{i1} f_{i1},\\
& \Delta(f_{i1}) \Delta(e_{i1}) = (f_{i1} \otimes 1+ K_{i} \otimes f_{i1}) \, (e_{i1}\otimes K_{i}^{-1} + 1 \otimes e_{i1}) \\
&\ \ =f_{i1} e_{i1} \otimes K_{i}^{-1} + f_{i1} \otimes e_{i1} + q_{(i)}^{2}\, e_{i1} K_{i} \otimes f_{i1} K_{i}^{-1} + K_{i} \otimes f_{i1} e_{i1}. 
\end{aligned}
\end{equation*}

\vskip 2mm 

Therefore we obtain 
\begin{equation*}
\begin{aligned}
& \Delta \left(\tau_{i0}(e_{i1} f_{i1} - f_{i1} e_{i1}) + \tau_{i1} (K_{i}^{-1} - K_{i})   \right) \\
& \ \ = \tau_{i0}(e_{i1} f_{i1} - f_{I1} e_{i1}) \otimes K_{i}^{-1} + \tau_{i0}(K_{i} \otimes (e_{i1}f_{i1} - f_{i1} e_{i1})\\ 
& \ \ \ \ \ + \tau_{i1} (K_{i}^{-1} \otimes K_{i}^{-1} - K_{i} \otimes K_{i})  \\
& \ \  = \tau_{i1}(K_{i} - K_{i}^{-1}) \otimes K_{i}^{-1} + K_{i} \otimes \tau_{i1} (K_{i}-K_{i}^{-1})\\
& \ \ \ \ \ + \tau_{i1}(K_{i}^{-1} \otimes K_{i}^{-1} - K_{i} \otimes K_{i})=0.
\end{aligned}
\end{equation*}

\vskip 2mm 

Suppose $k>1$ and set 
\begin{equation} \label{eq:A0}
\begin{aligned}
A_{0} &= \tau_{i1}(e_{i,k-1} f_{i, k-1} -f_{i, k-1} e_{i,k-1}) \\
&\phantom{=} + \tau_{i2} (e_{i, k-2} f_{i, k-2} K_{i}^{-1} - f_{i, k-2} e_{i, k-2} K_{i}) \\
& \phantom{=}+\cdots + \tau_{ik}(K_{i}^{-k+1} - K_{i}^{k-1}).
\end{aligned}
\end{equation} 

By Lemma \ref{lem:string}, $A_{0}=0$. Multiply $A_{0}$ by $K_{i}^{-1} + K_{i}$. Then we obtain 
\begin{equation} \label{eq:A}
A:=A_{0} (K_{i}^{-1} + K_{i}) = B+C=0,
\end{equation}
where 
\begin{equation} \label{eq:B}
\begin{aligned}
B &= \tau_{i1}(e_{i,k-1} f_{i, k-1}K_{i}^{-1}  -f_{i, k-1} e_{i,k-1} K_{i}) \\ 
&\phantom{=} + \tau_{i2} (e_{i, k-2} f_{i, k-2} K_{i}^{-2} - f_{i, k-2} e_{i, k-2} K_{i}^2) \\
& \phantom{=}+\cdots + \tau_{ik}(K_{i}^{-k} - K_{i}^{k}), 
\end{aligned}
\end{equation} 

\begin{equation} \label{eq:C}
\begin{aligned}
C &= \tau_{i1}(e_{i,k-1} f_{i, k-1}K_{i}  -f_{i, k-1} e_{i,k-1} K_{i}^{-1}) \\
& \phantom{=}+ \tau_{i2} (e_{i, k-2} f_{i, k-2}  - f_{i, k-2} e_{i, k-2}  \\
&\phantom{=} +\cdots + \tau_{ik}(K_{i}^{-k+2} - K_{i}^{k-2}). 
\end{aligned}
\end{equation} 

\vskip 2mm 

The induction hypothesis gives $\Delta(A_{0})=0$. Since $\Delta$ is an algebra homomorphism on $\widehat{U}$, we get 
\begin{equation} \label{eq:D}
\begin{aligned}
0&= \Delta(A_{0}) \Delta(K_{i}^{-1} + K_{i}) = \Delta(A_{0}(K_{i}^{-1}+K_{i})) \\
&=\Delta(A) = \Delta(B+C)=\Delta(B)+\Delta(C). 
\end{aligned}
\end{equation}

The relation \eqref{eq:delta} gives 
\begin{equation*}
\tau_{i0}(e_{ik} f_{ik} - f_{ik} e_{ik}) +B=0
\end{equation*}
and we would like to prove
\begin{equation*}
\Delta(\tau_{i0}(e_{ik}f_{ik} - f_{ik} e_{ik}) + B)
=\Delta(\tau_{i0}(e_{ik}f_{ik} - f_{ik} e_{ik})) + \Delta(B)=0.
\end{equation*}

\vskip 2mm 

Since $B+C=0$, we have 
$$\tau_{i0}(e_{ik} f_{ik} - f_{ik} e_{ik}) = -B =C,$$
which implies
$$\Delta(\tau_{i0}(e_{ik}f_{ik} - f_{ik} e_{ik})) = \Delta(C).$$

\vskip 2mm

Therefore, by \eqref{eq:D}, we obtain 
\begin{equation*}
\begin{aligned}
\Delta(\tau_{i0}(e_{ik}f_{ik} & - f_{ik} e_{ik}) + B)
=\Delta(\tau_{i0}(e_{ik}f_{ik} - f_{ik} e_{ik})) + \Delta(B)\\
&=\Delta(C)+\Delta(B)=0.
\end{aligned}
\end{equation*}

\vskip 2mm

{\bf Case 2\,:}  Suppose $k<l$.

\vskip 3mm 

In this case, by Lemma \ref{lem:string}, for all $k>0$, we have 
\begin{equation}\label{eq:delta-0}
\sum_{n=0}^{k} \tau_{in} \left(q_{(i)}^{-n(l-k)} e_{i, l-n}\, f_{i, k-n} \, K_{i}^{-n}  - q_{(i)}^{n(l-k)} f_{i, k-n}\, e_{i, l-n}\, K_{i}^{n}\right)=0. 
\end{equation}
We need to prove 
\begin{equation} \label{eq:delta-1}
\Delta\left(\sum_{n=0}^{k} \tau_{in} (q_{(i)}^{-n(l-k)}e_{i, k-n}\, f_{i, k-n} \, K_{i}^{-n}  -
q_{(i)}^{n(l-k)} f_{i, k-n}\, e_{i, k-n}\, K_{i}^{n}) \right)=0. 
\end{equation}

\vskip 2mm 

We will use induction on $k$. 

\vskip 2mm 

If $k=1$, $l>1$, we have
\begin{equation} \label{eq:induction}
\tau_{i0}(e_{il}f_{i1} -f_{i1}e_{il}) + \tau_{i1} e_{i, l-1}\,(q_{(i)}^{-l+1} K_{i}^{-1} -q_{(i)}^{l-1} K_{i})=0.
\end{equation}

We will verify
$$\Delta \left(\tau_{i0}(e_{il}f_{i1} -f_{i1}e_{il}) + \tau_{i1} e_{i, l-1}\,(q_{(i)}^{-l+1} K_{i}^{-1} -q_{(i)}^{l-1} K_{i})\right)=0$$
by a direct calculation. 

\vskip 3mm 

Recall that 
\begin{equation*}
\begin{aligned}
\Delta(e_{il})&=e_{il} \otimes K_{i}^{-l} + q_{(i)}^{l-1} e_{i, l-1} \otimes K_{i}^{-l+1} e_{i1} + + q_{)i)}^{2(l-2)} e_{i,l-2} \otimes K_{i}^{-l+2} e_{i,2}  \\ 
& \ \  + \cdots + q_{(i)}^{2(l-2)} e_{i2} \otimes K_{i}^{-2} e_{i, l-2} + q_{(i)}^{l-1} e_{i1} \otimes K_{i}^{-1} e_{i,l-1} + 1 \otimes e_{il},\\
\Delta(f_{i1})& = f_{i1} \otimes 1 + K_{i} \otimes f_{i1}.
\end{aligned}
\end{equation*}

\vskip 2mm 

Thus we have 
\begin{equation*}
\begin{aligned}
\tau_{i0} & \left(\Delta(e_{il} )\Delta(f_{i1})  -\Delta(f_{i1}) \Delta(e_{il}) \right) \\
& \ \  =\tau_{i0} (e_{il} f_{i1} - f_{i1} e_{il}) \otimes K_{i}^{-l} \\
& \ \ \phantom{=}+ q_{(i)}^{l-1} \, \tau_{i0} (e_{i,l-1}f_{i1}-f_{i1} e_{i,l-1}) \otimes K_{i}^{-l+1} e_{i,1}  \\
%& \ \ + q_{(i)}^{2(l-2)}  \tau_{i0}(e_{i,l-2} f_{i1} - f_{i1} e_{i,l-2}) \otimes K_{i}^{-l+2} e_{i2} \\
%& \ \ + \cdots + q_{(i)}^{2(l-2)} \tau_{i0} (e_{i2} f_{i1} -f_{i1} e_{i2}) \otimes K_{i}^{-2} e_{i,l-2} \\
& \ \ \phantom{=}+ \cdots + q_{(i)}^{l-1} \tau_{i0} (e_{i1}f_{i1} - f_{i1} e_{i1} ) \otimes K_{i}^{-l} e_{i, l-1}\\
& \ \ \phantom{=}+ q_{(i)}^{l-1} e_{i,l-1} K_{i} \otimes \tau_{i0} (e_{i1} f_{i1} - f_{i1} e_{i1}) K_{i}^{-l+1} \\
& \ \ \phantom{=}+ e_{i,l-2} K_{i} \otimes \tau_{i0} (e_{i2} f_{i1} - f_{i1} e_{i2}) K_{-l+2}\\
%& \ \ + \cdots + q_{(i)}^{-2(l-4)} e_{i2} K_{i} \otimes \tau_{i0} (e_{i,l-2} f_{i1} - f_{i1} e_{i,l-2}) K_{i}^{-2} \\
& \ \ \phantom{=}+ \cdots + q_{(i)}^{-(l-3)} e_{i} K_{i} \otimes \tau_{i0} (e_{i,l-1} f_{i1} - f_{i1} e_{i, l-1}) K_{i}^{-1} \\
& \ \ \phantom{=}+ K_{i} \otimes \tau_{i0} (e_{il} f_{i1} - f_{i1} e_{il}). 
\end{aligned}
\end{equation*}

\vskip 2mm 

Using the relation \eqref{eq:induction}, we get 
%\begin{equation*}
%\begin{aligned}
%& \tau_{i0}  \left(\Delta(e_{il} )\Delta(f_{i1})  -\Delta(f_{i1}) \Delta(e_{il}) \right)
% = \tau_{i1} (e_{i, l-1} (q_{(i)}^{l-1} K_{i} - q_{(i)}^{-l+1} K_{i}^{-1})) \otimes K_{i}^{-l} \\
%& + q_{(i)}^{l-1} \tau_{i1} (e_{i.l-2}(q_{(i)}^{l-2} K_{i} - q_{(i)}^{-l+2} K_{i}^{-1})) \otimes K_{i}^{-l+1} e_{i,1}\\
%& + q_{(i)}^{2(l-2)} \tau_{i1}(e_{i,l-2}(q_{(i)}^{l-3} K_{i} - q_{(i)}^{-l+3} K_{i}^{-1}))\otimes K_{i}^{-l+2} e_{i2} \\
%& + \cdots + q_{(i)}^{2(l-2)} \tau_{i1} (e_{i1} (q_{(i)} K_{i} - q_{(i)}^{-1} K_{i}^{-1})) \otimes K_{i}^{-2} e_{i,l-2} \\
%&+ q_{(i)}^{l-1} \tau_{i1}(K_{i}-K_{i}^{-1}) \otimes K_{i}^{-1} e_{i,l-1} \\
%& + q_{(i)}^{l-1} e_{i,l-1} K_{i} \otimes \tau_{i1} (K_{i} - K_{i}^{-1}) K_{i}^{-l+1} \\
%& + e_{i, l-2} K_{i} \otimes \tau_{i1} (e_{i1}(q_{(i)} K_{i} - q_{(i)}^{-1} K_{i}^{-1})) K_{i}^{-l+2} \\
%& + \cdots + q_{(i)}^{-2(l-4)} e_{i2} K_{i} \otimes \tau_{i1}(e_{i,l-3} (q_{(i)}^{l-3} K_{i} - q_{(i)}^{-l+3} K_{i}^{-1})) %K_{i}^{-2} \\
%& + q_{(i)}^{-l+3} e_{i1} K_{i} \otimes \tau_{i1} (e_{i,l-2}(q_{(i)}^{l-2} K_{i} - q_{(i)}^{-l+3} K_{i}^{-1})) K_{i}^{-1} \\
%& + K_{i} \otimes \tau_{i1} (e_{i,l-1} (q_{(i)}^{l-1} K_{i} - q_{(i)}^{-l+1} K_{i}^{-1})),
%\end{aligned}
%\end{equation*}
%which is equal to 
\begin{equation}
\begin{aligned} \label{eq:S}
\tau_{i0}  & \left(\Delta(e_{il} )\Delta(f_{i1})  -\Delta(f_{i1}) \Delta(e_{il}) \right) \\
 =& \tau_{i1} (q_{(i)}^{l-1} e_{i,l-1} K_{i} \otimes K_{i}^{-l+2}  
 + q_{(i)} e_{i,l-2} K_{i} \otimes e_{i1} K_{i}^{-l+3} \\
& + \cdots + q_{(i)} e_{i,1} K_{i} \otimes e_{i,l-2} + q_{(i)}^{l-1} K_{i} \otimes e_{i,l-1} K_{i})\\ 
& -  \tau_{i1} (q_{(i)}^{-l+1} e_{i,l-1} K_{i}^{-1} \otimes K_{i}^{-l} + q_{(i)}^{-2l+3} e_{i,l-2} K_{i}^{-1} 
 \otimes e_{i1} K_{i}^{-l+1} \\
 & + \cdots q_{(i)}^{-2l+3} e_{I1} K_{i}^{-1} \otimes e_{i,l-2} K_{i}^{-2} 
 +q_{(i)}^{-l+1} K_{i}^{-1} \otimes e_{i,l-1} K_{i}^{-1}).
\end{aligned}
\end{equation}

\vskip 3mm 

Now we consider the co-multiplication of $e_{i,l-1}(q_{(i)}^{-l+1} K_{i}^{-1} - q_{(i)}^{l-1} K_{i})$. 
Note that 
$$\Delta(q_{(i)}^{-l+1} K_{i}^{-1} - q_{(i)}^{l-1} K_{i}) = q_{(i)}^{-l+1}(K_{i}^{-1} \otimes K_{i}^{-1})
- q_{(i)}^{l-1}(K_{i} \otimes K_{i}).$$

\vskip 2mm 

Hence we have 

\begin{equation} \label{eq:T}
\begin{aligned}
\Delta(\tau_{1} & \big (e_{i,l-1}(q_{(i)}^{-l+1} K_{i}^{-1} - q_{(i)}^{l-1} K_{i}))) \\
 =&  \tau_{i1} (q_{(i)}^{-l+1} e_{i,l-1} K_{i}^{-1} \otimes K_{i}^{-l} 
+ q_{(i)}^{-2l+3} e_{i,l-2} K_{i}^{-1} \otimes e_{i,1} K_{i}^{-l+1} \\
& + \cdots + q_{(i)}^{-2l+3} e_{i1} K_{i}^{-1} \otimes e_{i,l-2} K_{i}^{-2} 
+ q_{(i)}^{-l+1} K_{i}^{-1} \otimes e_{i,l-1} K_{i}^{-1} \big ) \\
& - \tau_{i1} \big (q_{(i)}^{l-1} e_{i,l-1} K_{i} \otimes K_{i}^{-l+2} 
+ q_{(i)} e_{i,l-2} K_{i} \otimes e_{i1} K_{i}^{-l+3}\\
& + \cdots + q_{(i)} e_{i1} K_{i} \otimes e_{i,l-2} 
+ q_{(i)}^{l-1} K_{i} \otimes e_{i,l-1} K_{i}\big ).
\end{aligned}
\end{equation}

Therefore, combining \eqref{eq:S} and \eqref{eq:T}.we obtain 
$$\Delta \left(\tau_{i0}(e_{il}f_{i1} -f_{i1}e_{il}) 
+ \tau_{i1} \, e_{i, l-1}\,(q_{(i)}^{-l+1} K_{i}^{-1} -q_{(i)}^{l-1} K_{i})\right)=0.$$

\vskip 2mm 

Assume that $k>1$. Our argument is similar to the case when $k=l$. We already know
\begin{equation*}
\begin{aligned}
\tau_{i0} & \left(e_{il} f_{ik} - f_{ik} e_{il} \right) + \tau_{i1} \left(q_{(i)}^{-(l-k)} e_{i, l-1} \, f_{i, k-1} K_{i}^{-1}
- q_{(i)}^{l-k} f_{i,k-1} \, e_{i, l-1} K_{i}\right) \\
& +  \tau_{i2} \left(q_{(i)}^{-2(l-k)} e_{i,l-2} \,f_{i, k-2} K_{i}^{-2} - q_{(i)}^{2(l-k)} f_{i,k-2} \,e_{i,l-2} K_{i}^2 \right) \\
& + \cdots + \tau_{ik}\left(q_{(i)}^{-k(l-k)} e_{i,l-k} K_{i}^{-1} - q_{(i)}^{k(l-k)} e_{i,l-k} K_{i}^{k}\right) =0,
\end{aligned} 
\end{equation*}
and we need to show $\Delta$ preserves the above relation. 

\vskip 2mm

Set 
\begin{equation*}
\begin{aligned}
A_{0} & = \tau_{i1} \left( e_{i, l-1}\, f_{i, k-1} 
-  f_{i,k-1} \, e_{i, l-1} \right) \\
&\phantom{=} +  \tau_{i2} \left(q_{(i)}^{-(l-k)} e_{i,l-2} \, f_{i, k-2} K_{i}^{-1} - q_{(i)}^{l-k} f_{i,k-2} \, e_{i,l-2} K_{i} \right) \\
& \phantom{=} + \cdots + \tau_{ik}\left(q_{(i)}^{-(k-1)(l-k)} e_{i,l-k} K_{i}^{-1} - q_{(i)}^{(k-1)(l-k)} e_{i,l-k} K_{i}^{k}\right),
\end{aligned} 
\end{equation*}
which is equal to $0$.

\vskip 2mm 

Multiply $A_{0}$ by $q_{(i)}^{-(l-k)} K_{i}^{-1} + q_{(i)}^{l-k} K_{i}$ to obtain 
$$A:=A_{0}\left(q_{(i)}^{-(l-k)} K_{i}^{-1} + q_{(i)}^{l-k} K_{i}\right)=B+C=0,$$
where 
\begin{equation*}
\begin{aligned}
B&=\tau_{i1} \left(q_{(i)}^{-(l-k)} e_{i, l-1} \, f_{i, k-1} K_{i}^{-1}
- q_{(i)}^{l-k} f_{i,k-1}\,  e_{i, l-1} K_{i}\right) \\
&\phantom{=} +  \tau_{i2} \left(q_{(i)}^{-2(l-k)} e_{i,l-2} \, f_{i, k-2} K_{i}^{-2} - q_{(i)}^{2(l-k)} f_{i,k-2} \, e_{i,l-2} K_{i}^2 \right) \\
&\phantom{=} + \cdots + \tau_{ik}\left(q_{(i)}^{-k(l-k)} e_{i,l-k} K_{i}^{-1} - q_{(i)}^{k(l-k)} e_{i,l-k} K_{i}^{k}\right),
\end{aligned}
\end{equation*}

\begin{equation*}
\begin{aligned}
C &=\tau_{i1} \left(q_{(i)}^{l-k} e_{i, l-1} \, f_{i, k-1} K_{i}
- q_{(i)}^{-(l-k)} f_{i,k-1} \, e_{i, l-1} K_{i}^{-1} \right) \\
& \phantom{=}+  \tau_{i2} \left(e_{i,l-2}\, f_{i, k-2} -  f_{i,k-2} \, e_{i,l-2}  \right) \\
& \phantom{=}+ \cdots + \tau_{ik}\left(q_{(i)}^{-(k-2)(l-k)} e_{i,l-k} K_{i}^{-1} - q_{(i)}^{(k-2)(l-k)} e_{i,l-k} K_{i}^{k}\right).
\end{aligned}
\end{equation*}

As in the case of $k=l$, since $B+C=0$, we have 
$$\tau_{i0}(e_{il} f_{ik} - f_{ik} e_{il})=C$$
and 
$$\Delta(\tau_{i0}(e_{il} f_{ik} - f_{ik} e_{il}))=\Delta(C).$$

\vskip 2mm 

By the induction hypothesis, we have 
\begin{equation*}
\begin{aligned}
\Delta(A) &= \Delta(B)  +  \Delta(C)=\Delta(B+C) =\Delta \left(A_{0}(q_{(i)}^{-(l-k)} K_{-1} + q_{(i)}^{l-k} K_{i}\right) \\
& = \Delta(A_{0}) \Delta\left(q_{(i)}^{-(l-k)} K_{-1} + q_{(i)}^{l-k} K_{i}\right) =0.
\end{aligned}
\end{equation*}

Therefore, 
\begin{equation*}
%\begin{aligned}
\Delta\left(\tau_{i0}(e_{il} f_{ik} - f_{ik} e_{il}) + B\right)  
= \Delta\left(\tau_{i0}(e_{il} f_{ik} - f_{ik} e_{il}) \right) + \Delta(B) = \Delta(C)+\Delta(B)=0,
%\end{aligned}
\end{equation*}
which proves our claim. 

\vskip 3mm 

{\bf Case 3\,:} If $k>l$, we can prove our claim almost in the same way as we did above.  

\vskip 2mm

Therefore, we obtain a co-multiplication on $U$
$$\Delta : U \longrightarrow U \otimes U$$
as desired. 
\end{proof}

\vskip 8mm

\section{Triangular decomposition}

\vskip 2mm

Let $U^{+}$ (resp. $U^{-}$) be the subalgebra of $U_{q}(\g)$ generated by $e_{il}$ (resp. $f_{il}$) for $(i,l) \in I^{\infty}$. We will denote by $U^{0}$ the subalgebra generated by $q^{h}$ $(h \in P^{\vee})$. It is easy to see that $U^{0} = \bigoplus_{h \in P^{\vee}} \Q(q) q^{h}$. We will prove that the quantum Borcherds-Bozec algebra has a {\it triangular decomposition}.

\vskip 3mm 

We first prove the following lemma.

\begin{lemma} \label{lem:triangular} %\hfill
%\vskip 2mm 
{\rm Let $U^{\ge 0}$ (resp. $U^{\le 0}$) be the subalgebra of 
$U_{q}(\g)$ generated by $U^{0}$ and $U^{+}$ (resp. $U^{-}$ and $U^{0}$). 
Then we have the isomorphisms: 
\begin{equation}
 U^{\le 0}  \cong U^{-} \otimes U^{0}, \qquad
U^{\ge 0} \cong  U^{0} \otimes U^{+}.
\end{equation}
}
\end{lemma}

\vskip 1mm 

\begin{proof} \ 
We will prove the first  isomorphism only because the second one would follow from a similar argument.  

\vskip 2mm

Since $U^{-}$ is spanned by the monomials in $f_{il}$ $((i,l) \in I^{\infty})$, we can extract a monomial basis ${\mathbf B}^{-} = \{f_{\tau} \mid \tau \in \Omega \}$ of $U^{-}$ indexed by an ordered set $\Omega$. By the defining relations \eqref{eq:Uhat}, we have a surjective homomorphism
$$ U^{-} \otimes U^{0} \longrightarrow U^{\le 0}$$
given by 
$$f_{\tau} \otimes q^{h} \longmapsto f_{\tau} q^{h} \quad  (\tau \in \Omega, \  h \in P^{\vee}).$$  
\vskip 2mm
\noindent
Hence we need to show $f_{\tau} q^{h}$ $(\tau \in \Omega, h \in P^{\vee})$ are linearly independent. 

\vskip 3mm

Note that ${\mathbf B}^{-}$ can be decomposed into a disjoint union ${\mathbf B}^{-} = \bigsqcup_{\beta \in Q_{+}} {\mathbf B}_{-\beta}$, where ${\mathbf B}_{-\beta}$ consists of the monomials $f_{\tau}$ with $\text{deg}\, f_{\tau} = - \beta$. 

\vskip 3mm 

Now consider the linear dependence relation 
\begin{equation} \label{eq:Ule0}
\sum_{\tau, h} c_{\tau,h} f_{\tau} q^{h} =0 \qquad \text{with} \ \ \tau \in \Omega, \, h \in P^{\vee}, \, c_{\tau, h} \in \Q(q).
\end{equation}

\vskip 2mm 
\noindent
By the above observation, the relation \eqref{eq:Ule0} can be written as 
$$\sum_{\beta \in Q_{+}} \left( \sum_{ \substack{\text{deg}\, f_{\tau}=-\beta \\ h \in P^{\vee}}} c_{\tau, h} f_{\tau} q^{h}   \right) = 0,$$
which yields 
\begin{equation}\label{eq:Ule0-1}
\sum_{ \substack{\text{deg}\, f_{\tau}=-\beta \\ h \in P^{\vee}}} c_{\tau, h} f_{\tau} q^{h} =0 \ \ \text{for all} \ \beta \in Q_{+}.
\end{equation}

\vskip 2mm 

Let us write $f_{\tau} = f_{i_1, l_1} \cdots f_{i_r, l_r}$ with $l_{1} \alpha_{i_1} + \cdots + l_{r} \alpha _{i_r} = \beta$ and recall that 
\begin{equation*}
\begin{aligned}
\Delta(f_{il})  = & \sum_{m+n =l} q_{(i)}^{-mn} f_{im} K_{i}^{n} \otimes f_{in} \\
 = & f_{il} \otimes 1 + q_{(i)}^{-(l-1)} f_{i, l-1} K_{i} \otimes f_{i,1} + q_{(i)}^{-2(l-2)} f_{i,l-2} K_{i}^{2} \otimes f_{i,2} \\
& + \cdots  + \, q_{(i)}^{-(l-1)} f_{i,1} K_{i}^{l-1} \otimes f_{i,l-1}  + K_{i}^{l} \otimes f_{i,l}.
\end{aligned}
\end{equation*}
Hence we may write 
\begin{equation*}
\begin{aligned}
\Delta(f_{\tau}) = & f_{i_1, l_1} \cdots f_{i_r, l_r} \otimes 1 + \cdots + ( \text {intermediate terms})  \\
& + \cdots + K_{i_1}^{l_1} \cdots K_{i_r}^{l_r} \otimes f_{i_1, l_1} \cdots f_{i_r, l_r},\\
= & f_{\tau} \otimes 1 + \cdots + ( \text {intermediate terms}) + \cdots + q^{h_{\tau}} \otimes f_{\tau},
\end{aligned}
\end{equation*}
where $q^{h_{\tau}} = K_{i_1}^{l_1} \cdots K_{i_r}^{l_r}$. Applying the co-multiplication $\Delta$ in \eqref{eq:Ule0-1}, we obtain 
\begin{equation*}
\begin{aligned}
0 = & \sum_{\tau, h} c_{\tau, h} \Delta(f_{\tau})(q^h \otimes q^h) \\
= & \sum_{\tau, h} c_{\tau, h} \left(f_{\tau} q^{h} \otimes q^{h} +  ( \text{intermediate terms}) + q^{h_{\tau}+h} \otimes f_{\tau} q^{h}\right).
\end{aligned}
\end{equation*}

Let us focus on the terms of bi-degree $(0, -\beta)$:
\begin{equation*}
\begin{aligned}
0 = & \sum_{\tau, h} c_{\tau, h} q^{h_{\tau} + h} \otimes f_{\tau} q^{h} = \sum_{h}\left(\sum_{\tau} c_{\tau, h} q^{h_{\tau} + h} \otimes f_{\tau} q^{h} \right) \\
= & \sum_{h} \left(q^{h_{\tau} + h} \otimes \left(\sum_{\tau} c_{\tau, h} f_{\tau} q^{h}\right)\right).
\end{aligned}
\end{equation*}

\vskip 2mm 
\noindent
Since $q^{h_{\tau} + h}$ $(h \in P^{\vee})$ are linearly independent, by an elementary property of tensor product, we conclude
$$\sum_{\tau} c_{\tau, h} f_{\tau} q^{h} =0 \ \ \text{for all} \ h \in P^{\vee},$$
which implies $\sum_{\tau} c_{\tau, h} f_{\tau} =0$. But $f_{\tau}$ $(\tau \in \Omega)$ are linearly independent. Hence $c_{\tau, h}=0$  for all $\tau \in \Omega$, $h \in P^{\vee}$ as desired.
\end{proof}

\vskip 2mm

We prove our main theorem in this section.

\vskip 3mm 

\begin{theorem} \label{thm:triangular} 
{\rm The quantum Borcherds-Bozec algebra $U_{q}(\g)$ has the following triangular decomposition:
\begin{equation} \label{eq:triangular}
U_{q}(\g) \cong U^{-} \otimes U^{0} \otimes U^{+}.
\end{equation}
}
\end{theorem}

\begin{proof} \  We first show that there exists a surjective homomorphism 
\begin{equation} \label{eq:hom}
U^{-} \otimes U^{0} \otimes U^{+} \longrightarrow U_{q}(\g).
\end{equation}

That is, every element $u \in U_{q}(\g)$ can be written as 
\begin{equation} \label{eq:surjective}
u = \sum u^{-} u^{0} u^{+},
\end{equation}
where $u^{0} \in U^{0}$, $u^{\pm} \in U^{\pm}$. 

\vskip 3mm

By the defining relations of $U_{q}(\g)$, we have only to verify that $e_{jl} \, f_{ik}$  for $(i,k), (j,l) \in I^{\infty}$ can be written in the form of \eqref{eq:surjective}.   If $i \neq j$, we have seen that $e_{jl}\,f_{ik} = f_{ik}\, e_{jl}$ and \eqref{eq:surjective} is verified. So we will focus on the case when $i=j$. 

\vskip 3mm 

As in Section 2, we denote by $L$ and $R$ the left-hand side and the right-hand side of the relation \eqref{eq:string}\,:
\begin{equation*}
\begin{aligned}
& L=\sum_{\substack{m+n=k \\ n+s=l}} q_{(i)}^{n(m-s)} \tau_{in} \, e_{is} \, f_{im} \, K_{i}^{-n},\\
& R= \sum_{\substack{m+n=k \\ n+s=l}} q_{(i)}^{-n(m-s)} \tau_{in} \, f_{im}\, e_{is} \, K_{i}^{n}. 
\end{aligned}
\end{equation*}

Note that $$e_{is} K_{i}^{n} = q_{i}^{-nsa_{ii}} e_{is}= q_{(i)}^{-2ns} K_{i}^{n} e_{is}.$$ Hence 
$$ R= \sum_{\substack{m+n=k \\ n+s=l}} q_{(i)}^{-n(m+s)} \tau_{in} \, f_{im}\, K_{i}^{n} \, e_{is},$$
which is in the form of \eqref{eq:surjective}.

\vskip 2mm 

For all $k,l>0$, we will show that 
$$e_{il} \, f_{ik} = R -S,$$
where $S$ is in the form of \eqref{eq:surjective}. 

\vskip 2mm

Since 
$$e_{is}\, K_{i}^{-n} = q_{i}^{ns a_{ii}} K_{i}^{_n}\, e_{is} = q_{(i)}^{2ns} K_{i}^{-n}\, e_{is},$$
we have 
$$L=\sum_{\substack{m+n=k \\ n+s=l}} q_{(i)}^{-n(m-s)} \tau_{in} \,K_{i}^{-n} \, e_{is} \, f_{im}.$$

\vskip 3mm

{\bf Case 1\,:} Suppose $k=l$. 

\vskip 2mm 

In this case, 
\begin{equation*}
\begin{aligned}
L = & \sum_{n=0}^{k} \tau_{in}\, K_{i}^{-n} \, e_{i, k-n}\, f_{i,k-n}\\
= & \tau_{i0}\, e_{ik} \, f_{ik} + \tau_{i1} \, K_{i}^{-1}\, e_{i,k-1}\, f_{i,k-1} 
+ \tau_{i2} \, K_{i}^{-2} e_{i, k-2}\, f_{i,k-2}  \\
& + \cdots + \tau_{i,k-1} \, K^{-k+1}\, e_{i1}\, f_{i1} + \tau_{ik}\, K_{i}^{-k}.
\end{aligned}
\end{equation*}

\vskip 2mm 

We will prove our assertion by induction on $k$. 

\vskip 1mm 
If $k=1$, we have  
$$L=\tau_{i0}\, e_{i1}\, f_{i1} + \tau_{i1} \, K_{i}^{-1} =R,$$
which implies
$$e_{i1}\, f_{i1} = \frac{1}{\tau_{i0}} \left(R-\tau_{i1}\, K_{i}^{-1} \right) \in U^{-} \, U^{0}\, U^{+},$$
as desired. 

\vskip 2mm 

If $k > 1$, set 
$$S=\tau_{i1} \, K_{i}^{-1} \, e_{i, k-1} \, f_{i, k-1} + \cdots + 
\tau_{i, k-1} \, K^{-k+1} \, e_{i1}\, f_{i1} + \tau_{ik}\, K_{i}^{-k},$$
so that 
$$L=\tau_{i0}\, e_{ik}\, f_{ik} + S = R.$$
\vskip 2mm 

By induction hypothesis, all $e_{i1}f_{i1}$,  $e_{i2}f_{i2}$, $\ldots$,  $e_{i, k-1}f_{i, k-1}$ can be written 
in the form of  \eqref{eq:surjective}. Thus we can verify that $S \in U^{-}\, U^{0}\, U^{+}$, which yields
$$e_{ik}\, f_{ik} = \frac{1}{\tau_{i0}} (R-S) \in U^{-}\, U^{0}\, U^{+}.$$

\vskip 3mm 

{\bf Case 2\,:} Suppose $k<l$. 

\vskip 2mm 

In this case, 
\begin{equation*}
\begin{aligned}
L = & \sum_{n=0}^{k} q_{(i)}^{n(l-k)} \tau_{in}\, K_{i}^{-n} \, e_{i, l-n}\, f_{i,k-n}\\
= & \tau_{i0}\, e_{il} \, f_{ik} + q_{(i)}^{l-k} \tau_{i1} \, K_{i}^{-1}\, e_{i,l-1}\, f_{i,k-1} 
+ q_{(i)}^{2(l-k)} \tau_{i2} \, K_{i}^{-2} e_{i, l-2}\, f_{i,k-2}  \\
& + \cdots + q_{(i)}^{(k-1)(l-k)}\tau_{i,k-1} \, K^{-k+1}\, e_{i, l-k+1}\, f_{i1} 
+ q_{(i)}^{k(l-k)} \tau_{ik}\, K_{i}^{-k} \, e_{i,l-k}.
\end{aligned}
\end{equation*}

\vskip 2mm 

If $k=1$, then 
$$L=\tau_{i0} \, e_{il} \, f_{i1} + q_{(i)}^{l-1} \tau_{i1} \, K_{i}^{-1} e_{i, l-1} =R,$$
which implies 
$$e_{i,l-1}\, f_{i1} = \frac{1}{\tau_{i0}} \left(R-q_{(i)}^{l-1} \tau_{i1}\, K_{i}^{-1}\, e_{i, l-1}\right) \in U^{-}\, U^{0}\, U^{+}.$$

\vskip 2mm 

Assume that $k>1$ and set 
$$ S=q_{(i)}^{l-k} \tau_{i1} \, K_{i}^{-1}\, e_{i,l-1}\, f_{i,k-1} 
+ \cdots + q_{(i)}^{k(l-k)} \tau_{ik}\, K_{i}^{-k} \, e_{i,l-k}.$$
Then using the insduction hypothesis, we conclude 
$$e_{il}\, f_{ik} = \frac{1}{\tau_{i0}} \left(R-S \right) \in U^{-}\, U^{0}\, U^{+}.$$

\vskip 2mm

{\bf Case 3\, :} When $k>l$, we can prove our assertion  using the same argument as above.

\vskip 3mm 

We now prove the injectivity of the homomorphism \eqref{eq:hom}.
Let ${\mathbf B}^{+}=\{e_{\tau} \mid \tau \in \Omega\}$ denote a monomial basis of $U^{+}$.  We need to show that the set ${\mathbf B} = \{ f_{\tau} q^h e_{\mu} \mid \tau, \mu \in \Omega, h \in P^{\vee} \}$ is linearly independent. As in Lemma \ref{lem:triangular}, we have only to consider the linear dependence relation 
\begin{equation} \label{eq:dependence}
\sum_{\substack{h \in P^{\vee} \\ \text{deg}\, f_{\tau} + \text{deg}\, e_{\mu} = \gamma}} c_{\tau, h, \mu} f_{\tau} q^h e_{\mu} =0
\end{equation}  
for all $\gamma \in Q$. 

\vskip 2mm 

Write 
\begin{equation*}
\begin{aligned}
& \Delta(e_{\mu}) = e_{\mu} \otimes q^{-h_{\mu}} + (\text{intermediate terms}) + 1 \otimes e_{\mu}, \\
& \Delta(f_{\tau}) = f_{\tau} \otimes 1 + (\text{intermediate terms}) + q^{h_{\tau}} \otimes f_{\tau}
\end{aligned}
\end{equation*}
so that we have 
\begin{equation*}
\begin{aligned}
0  = & \Delta \left(\sum_{\substack{h \in P^{\vee} \\ \text{deg}\, f_{\tau} + \text{deg}\, e_{\mu} = \gamma}} c_{\tau, h, \mu} f_{\tau} q^h e_{\mu} \right) \\
 = & \sum_{ \substack{h \in P^{\vee} \\ \text{deg}\, f_{\tau} + \text{deg}\, e_{\mu} = \gamma}} c_{\tau, h, \mu} \left(f_{\tau} \otimes 1 + (\text{intermediate terms}) + q^{h_{\tau}} \otimes f_{\tau} \right) \\
 & \qquad \times (q^h \otimes q^h) (e_{\mu} \otimes q^{-h_{\mu}} + (\text{intermediate terms}) + 1 \otimes e_{\mu}).
\end{aligned}
\end{equation*}

\vskip 2mm 

Take a total ordering $\le$ on $Q$ given by the height and lexicographic ordering. Let $\Omega_{0}$ (resp. $\Omega_{1}$) be the set of all $\tau \in \Omega$ (resp. $\mu \in \Omega$) such that $\text{deg}\, f_{\tau}$ (resp. $\text{deg}\, e_{\mu}$) is minimal (resp. maximal) among the terms appearing in \eqref{eq:dependence} with respect to $\le$.  Since $\text{deg}\, f_{\tau} \in Q_{-}$, $\text{deg}\, e_{\mu} \in Q_{+}$ and $\text{deg}\, f_{\tau} + \text{deg}\, e_{\mu} = \gamma$, it is clear that $\tau \in \Omega_{0}$ if and only if $\mu \in \Omega_{1}$. 

\vskip 3mm

We now focus on the terms of bi-degree $(\text{max}, \text{min})$ in \eqref{eq:dependence}, which gives
\begin{equation*}
0  = \sum_{\substack {h \in P^{\vee} \\ \tau \in \Omega_{0} \\ \mu \in \Omega_{1}}}    c_{\tau, h, \mu} q^{h_{\tau} + h} e_{\mu} \otimes f_{\tau} q^{h - h_{\tau}}  = \sum_{\substack{\tau \in \Omega_{0} \\ h \in P^{\vee}}} \left(\sum_{\mu \in \Omega_{1}} c_{\tau, h, \mu} q^{h_{\tau} + h} e_{\mu} \right) \otimes f_{\tau} q^{h-h_{\mu}}.
\end{equation*}

\vskip 2mm 
\noindent
Since $f_{\tau} q^{h-h_{\mu}}$ $(\tau \in \Omega_{0}, h\in P^{\vee})$   are linearly independent, we have 
$$\sum_{\mu \in \Omega_{1}} c_{\tau, h, \mu} q^{h_{\tau}+h} e_{\mu} =0.$$
Therefore $c_{\tau, h, \mu} q^{h_{\tau}+h}=0$ and hence $c_{\tau, h, u}=0$ for all $\tau \in \Omega_{0}$, $\mu \in \Omega_{1}$.

\vskip 2mm 

 Repeat this process along with the total ordering $\le$ on $Q$, we conclude $c_{\tau, h, \mu}=0$ for all $h \in P^{\vee}$, $\tau, \mu \in \Omega$, which proves our theorem. 
\end{proof}

\vskip 3mm

\begin{remark} \
{\rm The surjectivity of the homomorphism \eqref{eq:hom} can be proved using \cite[Proposition 3.10 (ii)]{Bozec2014c}. In \cite[Remark 3.23]{Bozec2014c}, it was mentioned that the quantum Borcherds-Bozec algebras have a triangular decomposition.}
\end{remark}

\vskip 8mm

\section{Highest weight representation theory}

\vskip 2mm

Let $U_{q}(\g)$ be a quantum Borcherds-Bozec algebra and let $M$ be a $U_{q}(\g)$-module. We say that $M$ has
a {\it weight space decomposition} if
$$M = \bigoplus_{\mu \in P} M_{\mu}, \ \ \text{where}
 \ M_{\mu} = \{ m \in M \mid q^{h} \,  m = q^{\langle h, \mu \rangle} m \
\text{for all} \ h \in P^{\vee} \}.$$ 

We denote  $\text{wt}(M):=\{\mu \in
\h^* \mid M_{\mu} \neq 0 \}$. When $\text{dim} \, M_{\mu} < \infty$ for all $\mu \in P$, we define the {\it character} of $M$ to be 
$$\text{ch} (M) = \sum_{\mu \in P} (\dim \, M_{\mu}) e^{\mu},$$
where $e^{\mu}$ $(\mu \in P)$ are multiplicative basis vectors of the group algebra of $P$. 
A non-zero vector $m \in M_{\mu}$ is said to be {\it of weight $\mu$}. If $m$ is annihilated by all $e_{il}$ $((i,l) \in I^{\infty})$, $m$ is called a {\it maximal  vector} of weight $\mu$. 

\vskip 3mm

A $U_{q}(\g)$-module $V$ is called a {\it highest weight module with highest weight $\lambda$} if there is a non-zero vector $v_{\lambda}$ in $V$ such that 
\begin{enumerate}
\item[(i)] $q^{h} \,  v_{\lambda} = q^{\langle h, \lambda \rangle} v_{\lambda}$ for all $h \in P^{\vee}$,

\item[(ii)] $e_{il} \, v_{\lambda} = 0$ for all $(i,l) \in I^{\infty}$, 

\item[(iii)] $V=U_{q}(\g) v_{\lambda}$. 
\end{enumerate}

\vskip 2mm

Such a vector $v_{\lambda}$ is called a {\it highest weight vector} with highest weight $\lambda$. 
Note that $V_{\lambda} = \Q(q) v_{\lambda}$ and $V$ has a weight space decomposition $V = \bigoplus_{\mu \le \lambda} V_{\mu}$. If a $U_{q}(\g)$-module $M$ has a weight space decomposition, a maximal vector of weight $\lambda$ would generate a highest weight submodule with highest weight $\lambda$.  

\vskip 3mm

For $\lambda \in P$, let $J(\lambda)$ be the left ideal of
$U_{q}(\g)$ generated by the elements $q^{h}-q^{\langle h, \lambda \rangle}
\mathbf{1}$ $(h \in P^{\vee})$ and $e_{il}$  $((i,l) \in I^{\infty})$. 
Set $M(\lambda):= U_{q}(\g) \big / J(\lambda)$. Then $M(\lambda)$ becomes a $U_{q}(\g)$-module, called the {\it Verma module}, via left multiplication. The following properties of $M(\lambda)$ are straightforwad consequences of the definition. 

 \vskip 3mm
 
 \begin{proposition} \label{prop:Verma} \  
 {\rm The Verma module $M(\lambda)$ satisfies  the following properties. 
 
 \begin{enumerate}
 
 \item[(a)] $M(\lambda)$ is a highest weight module with highest weight $\lambda$.

 \item[(b)] $M(\lambda)$ has a unique maximal submodule.

 \item[(c)] $M(\lambda)$ is a free $U^{-}$-module of rank 1.

 \item[(d)] Every highest weight module with highest weight $\lambda$ is a quotient module of $M(\lambda)$. 
\end{enumerate} 
}
 \end{proposition}
 
 \vskip 2mm 
 
Let $R(\lambda)$ be the unique maximal submodule of $M(\lambda)$ and let $V(\lambda):=M(\lambda) \big / R(\lambda)$ the irreducible quotient of $M(\lambda)$, which is also a highest weight module with highest weight $\lambda$. The following proposition is one of the most important ingredients of the integrable representation theory of quantum Borcherds-Bozec algebras. 

\vskip 3mm

\begin{proposition} \label{prop:hw} %\hfill
%\vskip 2mm
{\rm Let $\lambda \in P^{+}$ be a dominant integral weight and let $V(\lambda) = U_{q}(\g) \, v_{\lambda}$ be the irreducible highest weight module with highest weight $\lambda$ and highest weight vector $v_{\lambda}$. 

Then the followng statements hold. 

%\vskip 2mm 
\begin{enumerate}
\item[(a)] If $i \in I^{\text{re}}$, then $f_{i}^{\langle h_{i}, \lambda \rangle +1} v_{\lambda} =0$.   

%\vskip 2mm

\item[(b)] If $i \in I^{\text{im}}$ and $\langle h_{i}, \lambda \rangle =0$, then $f_{ik} \, v_{\lambda}=0$ for all $k>0$. 
\end{enumerate}}
\end{proposition}
 
 \begin{proof} 
 
 (a) If $i \in I^{\text{re}}$, by the $U_{q}(sl_{2})$-representation theory, it is known that  $e_{i} f_{i}^{\langle h_{i}, \lambda \rangle +1} v_{\lambda} =0$. If $(j,l) \neq i$, then $j \neq i$ and by \eqref{eq:commute}, we have 
 $$e_{jl} \, f_{i}^{\langle h_{i}, \lambda \rangle + 1} v_{\lambda} = f_{i}^{\langle h_{i}, \lambda \rangle + 1} \, e_{jl} \, v_{\lambda} =0.$$
 Hence if $f_{i}^{\langle h_{i}, \lambda \rangle + 1} v_{\lambda} \neq 0$, it is a maximal  vector and would generate a highest weight submodule of $V(\lambda)$ with highest weight $\lambda - (\langle h_{i}, \lambda \rangle + 1)\alpha_{i} < \lambda$. Since $V(\lambda)$ is irreducible, this is a contradiction.  
 
 \vskip 3mm 
 
 (b) For each $(i,k) \in I^{\infty}$, we will show that $e_{jl} \, f_{ik}\, v_{\lambda} =0$ for all $(j,l) \in I^{\infty}$. 
 
 \vskip 3mm 
 
  If $j \neq i$, as in \eqref{eq:commute}, we have 
 \begin{equation}\label{eq:distinct}
 e_{jl} \, f_{ik} v_{\lambda} = f_{ik} \, e_{jl} \, v_{\lambda} =0 \ \ \text{for all} \ l>0 .
 \end{equation}
 
 \vskip 2mm 
 
 So we will concentrate on the case when $j=i$.  In this case, Lemma \ref{lem:string} yields 
 \begin{equation} \label{eq:hw-0} 
\sum_{\substack{m+n=k \\ n+s=l}} q_{(i)}^{n(m-s)} \tau_{in} \, e_{is} \, f_{im} \, K_{i}^{-n}
= \sum_{\substack{m+n=k \\ n+s=l}} q_{(i)}^{-n(m-s)} \tau_{in} \, f_{im}\, e_{is} \, K_{i}^{n}. 
\end{equation}

\vskip 1mm 

As in Lemma \ref{lem:string}, let $L$ (resp. $R$) denote the left-hand side (resp. right-hand side) of the above equation.  
We will prove our claim in 3 steps. Note that, since $\langle h_{i}, \lambda \rangle =0$, we have $K_{i}^{\pm 1} \ v_{\lambda} = v_{\lambda}$. 

\vskip 3mm 

{\bf Step 1:}  \ Suppose $k=l$. In this case, we have $m=s$ and 
$$R(v_{\lambda}) =  \sum_{n=0}^{k} \tau_{in} \, f_{i, k-n} e_{i, k-n} \, v_{\lambda} = \tau_{ik} \, v_{\lambda}.$$ 

\vskip 1mm 

On the other hand, 
$$L(v_{\lambda}) = \sum_{n=0}^{k} \tau_{in} \, e_{i, k-n}\, f_{i,k-n}\, v_{\lambda} =  \sum_{n=0}^{k-1} \tau_{in} \, e_{i, k-n}\, f_{i,k-n}\, v_{\lambda} + \tau_{i,k}\, v_{\lambda}.$$
It follows that 
$$\sum_{n=0}^{k-1} \tau_{in} \, e_{i, k-n}\, f_{i,k-n}\, v_{\lambda} =0.$$

\vskip  2mm 
Hence by induction on $k$, we obtain
\begin{equation}
 \label{eq:hw-1}
 e_{ik}\, f_{ik} \, v_{\lambda} =0 \ \ \text{for all} \ k>0.
\end{equation}

\vskip 3mm

{\bf Step 2:} Suppose $k < l$. In this case, we have 
\begin{equation*}
R(v_{\lambda}) = \sum_{n=0}^{k} q_{(i)}^{-n(k-l)} \tau_{in}\, f_{i, k-n} \, e_{i, l-n} v_{\lambda}=0
\end{equation*}
because $l-n >0$ for all $n=0, 1, \ldots, k$, which implies
\begin{equation*}
L(v_{\lambda}) = \sum_{n=0}^{k} q_{(i)}^{n(k-l)} \tau_{in} \, e_{i, l-n}\, f_{i,k-n}\, v_{\lambda} =0. 
\end{equation*}

\vskip 1mm 

 If $k=1$, we have $l >1$ and  
 $$0 = \tau_{i,0} \,e_{i,l} f_{i,1} \, v_{\lambda} + q_{(i)}^{1-l} \tau_{i,1} \,e_{i, l-1} \, f_{i,0} \, v_{\lambda} = \tau_{i,0} \, e_{i,l} f_{i,1} v_{\lambda}.$$
 Hence $e_{i,l} \, f_{i,1} \, v_{\lambda} =0$ for $l>1$. Note that \eqref{eq:hw-1} gives $e_{i,1} \, f_{i,1} v_{\lambda} =0$. Therefore, we get $e_{i,l}\, f_{i,1} \, v_{\lambda} =0$ for all $l\ge 1$, which implies $f_{i,1} \, v_{\lambda} =0$, for otherwise, it would generate a highest weight submodule with highest weight $\lambda - \alpha_{i} < \lambda$. 
 
 \vskip 3mm 
 
 Assume that 
 \begin{equation} \label{eq:hypo}
 f_{i,1} \, v_{\lambda} =0, \ldots, f_{i, k-1} v_{\lambda} =0.
 \end{equation}

  \vskip 2mm
 
 Then we have 
 \begin{equation*}
 \begin{aligned}
 0 = &  L(v_{\lambda}) =  \sum_{n=0}^{k} q_{(i)}^{n(k-l)} \tau_{in} \, e_{i, l-n}\, f_{i,k-n}\, v_{\lambda} \\
  = & \tau_{i,0}\, e_{i,l}\, f_{i,k} \, v_{\lambda} + q_{(i)}^{k-l} \tau_{i,1} \, e_{i, l-1}\, f_{i, k-1} \, v_{\lambda} + q_{(i)}^{2(k-l)} \tau_{i,2}\, e_{i, l-2}\, f_{i, k-2}\, v_{\lambda} \\
 & + \cdots + q_{(i)}^{(k-1)(k-l)} \tau_{i,k-1}\, e_{i, l-k+1} \, f_{i,1} v_{\lambda} + q_{(i)}^{k(k-l)} \tau_{i,k}  \, e_{i,l} \, f_{i,0} \, v_{\lambda} \\
 = & \tau_{i,0} \, e_{i,l} \,f_{i,k} \, v_{\lambda}.
 \end{aligned}
 \end{equation*}  
 Therefore, combined with \eqref{eq:hw-1}, we obtain 
 $$e_{i,l} \, f_{i,k} \,v_{\lambda} =0 \ \ \text{for all} \ l\ge k.$$
 
 \vskip 2mm 
 
 {\bf Step 3:} Suppose $k>l$. Then, since $k-l < k$,  by the induction hypothesis \eqref{eq:hypo}, the relation \eqref{eq:hw-0} implies 
 $$R(v_{\lambda}) = \sum_{n=0}^{l} q_{(i)}^{-n(k-l)} \tau_{i,n}\, f_{i,k-n}\, e_{i, l-n} \, v_{\lambda} = q_{(i)}^{-l(k-l)} f_{i,k-l}\, v_{\lambda} =0.$$
 
 \vskip 1mm
  
On the other hand, by the induction hypothesis \eqref{eq:hypo} again, 
we have 
\begin{equation*}
\begin{aligned}
L(v_{\lambda}) & = \sum_{n=0}^{l} q_{(i)}^{n(k-l)} \tau_{i,n} \, e_{i, l-n}\, f_{i, k-n}\, v_{\lambda} \\
& = \tau_{i,0} \, e_{i,l} \, f_{i,k} \, v_{\lambda} + q_{(i)}^{k-l} \tau_{i,1} \, e_{i, l-1} \, f_{i, k-1} \, v_{\lambda} + \cdots + q_{(i)}^{l(k-l)} \tau_{i,l} \, e_{i,0} \, f_{i, k-l} \, v_{\lambda} \\
& = \tau_{i,0}\, e_{i,l} \, f_{i,k} \, v_{\lambda} =0, 
\end{aligned}
\end{equation*}
which yields $e_{i,l} \, f_{i,k} \, v_{\lambda} =0$ for all $l<k$. 

\vskip 3mm 

Therefore, combing \eqref{eq:distinct}, (Step 1) and (Step 2), we obtain
$$e_{j,l} \, f_{i,k} \, v_{\lambda} =0 \ \ \text{for all} \ (j,l) \in I^{\infty}.$$
Hence by induction on $k$, we conclude $f_{i,k} \, v_{\lambda} =0$ for all $k >0$,
which proves our claim.   
\end{proof}
 
\vskip 3mm 

\begin{example} \label{ex:string} \ {\rm In this example, we briefly describe the structure of the irreducible highest weight module $V(\lambda)$ over the quantum string algebra $U_{(i)}$ for $i\in I^{\text{im}}$. If $\langle h_{i}, \lambda \rangle =0$, $V(\lambda)$ is the 1-dimensional trivial representation.

\vskip 2mm 

If $\langle h_{i}, \lambda \rangle >0$, then $V(\lambda)$ is isomorphic to the Verma module $M(\lambda)$ and ${\mathbf B}(\lambda):= \{b \, v_{\lambda} \mid b \in {\mathbf B} \}$ is a basis of $V(\lambda)$, where ${\mathbf B}$ is the basis of $U_{(i)}^{-}$ given in Example \ref{ex:iso} and Example \ref{ex:non-iso}. 
}
\end{example}

\vskip 3mm 

\begin{proposition} \label{prop:imaginary}  

{\rm Let $\lambda \in P^{+}$ be a dominant integral weight and let $V(\lambda)=U_{q}(\g)  v_{\lambda}$ be the irreducible highest weight $U_{q}(\g)$-module with highest weight $\lambda$ and highest weight vector $v_{\lambda}$. For $i \in I^{\text{im}}$ and $\mu \in \text{wt}(V(\lambda))$, the following statements hold. 

\vskip 2mm
\begin{enumerate}

\item[(a)] $\langle h_{i}, \mu \rangle \ge 0$. 

\item[(b)] If $\langle h_{i}, \mu \rangle  =0$, then $V(\lambda)_{\mu - l \alpha_{i}} =0$ \, for all $l>0$.

\item[(c)] If $\langle h_{i}, \mu \rangle  =0$, then $f_{il}(V(\lambda)_{\mu})=0$.

\item[(d)] If $\langle h_{i} , \mu \rangle \le -l a_{ii}$, then $e_{il}(V(\lambda)_{\mu})=0$. 
\end{enumerate}}
\end{proposition}

\begin{proof} \ (a) Since $\mu \in \text{wt}(V(\lambda))$, $\mu = \lambda - \beta$ for some $\beta \in Q_{+}$. Write $\beta =l_{1} \alpha_{i_1} + \cdots + l_{r} \alpha_{i_r}$. Then  since $a_{ij} \le 0$ for all $j \in I$, we have
$$\langle h_{i}, \mu \rangle = \langle h_{i}, \lambda \rangle - \langle h_{i}, \beta \rangle
= \langle h_{i}, \lambda \rangle -(l_{1} a_{i, i_1} + \cdots + l_{r} a_{i, i_r}) \ge 0. $$

\vskip 2mm 

(b) If $\langle h_{i}, \mu \rangle =0$, then $\langle h_{i}, \lambda \rangle =0$ and $a_{i, i_k}=0$  for all $k=1, \ldots, r$.
Let $u =f_{i_1, l_1} \cdots f_{i_r, l_r} $ be a monomial such that $u \, v_{\lambda} \in V(\lambda)_{\mu}$. Since $a_{i, i_k}=0$ for all $k=1, \ldots, r$, by the defining relation of quantum Borcherds-Bozec algebras, $f_{il}\, f_{i_k,l_k} = f_{i_k, l_k} \, f_{il} $. 
Hence $f_{il} u  \, v_{\lambda} = u \, f_{il} \, v_{\lambda} =0$ by Proposition  \ref{prop:hw}\,(b). 

\vskip 2mm 

(c) Our statement follows immediately from (b).

\vskip 2mm 

(d) Suppose $e_{il}(V(\lambda)_{\mu} \neq 0$. Then $\mu + l \alpha_{i} \in \text{wt}(V(\lambda))$ and by (a)
$$ 0 \le \langle h_{i}, \mu + l \alpha_{i} \rangle  = \langle h_{i} , \mu \rangle +  l a_{ii} \le 0, $$which implies $\langle h_{i}, \mu + l \alpha_{i} \rangle =0$. Then by (b), $\mu = (\mu+ l \alpha_{i}) - l \alpha_{i}$ would not be a weight of $V(\lambda)$, which is a contradiction. Hence $e_{il}(V(\lambda)_{\mu}  =0$.
\end{proof}

\begin{remark}
There is a sign error in \cite[Proposition 2.2 (b)]{K2019}.
\end{remark}

\vskip 10mm

\section{The category ${\mathcal O}_{\rm int}$}

\vskip 3mm

We introduce the notion of integrable representations.

\begin{definition}\label{def:Oint}
The {\it category ${\mathcal O}_{\text{\rm int}}$} consists of
$U_{q}(\g)$-modules $M$ such that
\begin{itemize}
\item [(i)] $M$ has a weight space decomposition $M=\bigoplus_{\mu
\in P} M_{\mu}$ with $\dim_{\Q(q)} M_{\mu} < \infty$ for all $\mu \in
P$.

\item[(ii)] there exist finitely many weights $\lambda_1, \ldots,
\lambda_s \in P$ such that $$\text{wt}(M) \subset \bigcup_{j=1}^s
(\lambda_j - Q_{+}),$$

\item[(iii)] if $i \in I^{\text{re}}$, $f_{i}$ is locally nilpotent
on $M$,

\item[(iv)] if $i \in I^{\text{im}}$, we have $\langle h_i, \mu
\rangle \ge 0$ for all $\mu \in \text{wt}(M)$,

\item[(v)] if $i \in I^{\text{im}}$ and $\langle h_{i}, \mu \rangle
=0$, then $f_{il}(M_{\mu})=0$ for all $l \in \Z_{>0}$,

\item[(vi)] if $i \in I^{\text{im}}$ and $\langle h_i, \mu \rangle
\le - l a_{ii}$, then $e_{il}(M_{\mu})=0$ for all $l \in \Z_{>0}$.
\end{itemize}
\end{definition}

\vskip 3mm

\begin{remark}\hfill 
{\rm
 
\begin{enumerate}

\item[(a)] By (ii), $e_{il}$ $((i,l)\in I^{\infty})$ are locally nilpotent.

\item[(b)] If $i \in I^{\text{im}}$,  $f_{il}$  are not necessarily locally
nilpotent.

\item[(c)] The irreducible highest weight $U_{q}(\g)$-module
$V(\lambda)$ with $\lambda \in P^{+}$ belongs to the category
${\mathcal O}_{\text{int}}$.

\item[(d)] A submodule or a quotient module of a $U_{q}(\g)$-module in  the category 
${\mathcal O}_{\text{int}}$ is again an object of ${\mathcal O}_{\text{int}}$.

\item[(e)] A finite number of direct sums or a finite number of tensor products of $U_{q}(\g)$-modules in the 
category ${\mathcal O}_{\text{int}}$ is again an object of ${\mathcal O}_{\text{int}}$.
\end{enumerate}}
\end{remark}

\vskip 3mm

Now we would like to give a character formula for $V(\lambda)$ with $\lambda \in P^{+}$. For this purpose, we need some preparation \cite{BSV2016}

\vskip 3mm

For a dominant integral weight $\lambda \in P^{+}$, let $F_{\lambda}$ be the set of elements
of the form $s= \sum_{k=1}^r l_k \alpha_{i_k}$ $(r \ge 0)$ such that
\begin{itemize}
\item[(i)] $i_{k} \in I^{\text{im}}$, $l_{k} \in \Z_{>0}$ for all $1 \le
k \le r$,

\item[(ii)] $(\alpha_{i_p}, \alpha_{i_q}) =0$ for all $1 \le p, q \le
r $,

\item[(iii)] $(\alpha_{i_k}, \lambda) =0$ for all $1 \le k \le r$.
\end{itemize}

\noindent When $r=0$, we understand $s=0$.

\vskip 3mm

For $s=\sum s_k \alpha_{i_k} \in F_{\lambda}$, we define
\begin{equation} \label{eq:sign}
\begin{aligned}
 d_{i}(s) & = \begin{cases} \# \{k \mid i_k =i \} & \text{if} \ i
\notin I^{\text{iso}}, \\
\sum_{i_k=i} s_k & \text{if} \ i \in I^{\text{iso}},
\end{cases} \\
\epsilon(s) & = \prod_{i \notin I^{\text{iso}}} (-1)^{d_i(s)}
\prod_{i \in I^{\text{iso}}} \phi(d_{i}(s))\\
& = (-1)^{\#(\text{supp}(s) \cap I \setminus I^{\text{iso}})}
\prod_{i \in I^{\text{iso}}} \phi(d_{i}(s)),
\end{aligned}
\end{equation}
where $\phi(n)$ are given by $\prod_{k=1}^{\infty} (1-q^k) = \sum_{n\ge 0} \phi(n) q^n.$

\vskip 3mm

Define
\begin{equation} \label{eq:slambda}
S_{\lambda} = \sum_{s \in F_{\lambda}} \epsilon(s) e^{-s}.
\end{equation}

\noindent
If $F_{\lambda} = \{0\}$, we understand $S_{\lambda}=1$. 

\vskip 3mm

\begin{example} \label{ex:rank1} \hfill

\vskip 2mm

{\rm  Let $i \in I^{\text{im}}$ and consider the quantum string algebra $U_{(i)}$. 

\vskip 2mm

(a) If $\langle h_{i}, \lambda \rangle >0$, then $F_{\lambda} =\{0\}$ and $S_{\lambda} =1$. 

\vskip 2mm 

(b) if $ i \in I^{\text{iso}}$ and $\langle h_{i}, \lambda \rangle =0$, then $\lambda =0$, $F_{0} = \{ l \alpha_{i} \mid l \ge 0 \}$ and $d_{i}(l \alpha_{i}) = l$. Hence
$$S_{0}=\sum_{l \ge 0} \epsilon(l \alpha_i) \, e^{-l\alpha_i} = \sum_{l
\ge 0} \phi(l) \, e^{-l \alpha_i} = \prod_{k=1}^{\infty} (1- e^{-k
\alpha_{i}}).$$

\vskip 2mm 

(c) If $i \in I^{\text{im}} \setminus I^{\text{iso}}$ and $\langle h_{i}, \lambda \rangle =0$, then $\lambda =0$, $F_{0} = \{ l \alpha_{i} \mid l \ge 0 \}$ and
\begin{equation*}
 d_i (l \alpha_i) = \begin{cases} 0 & \text{if} \ l=0, \\
1 & \text{if} \ l \ge 1,
\end{cases} \qquad 
\epsilon(l \alpha_i) = \begin{cases} 1 & \text{if} \ l=0, \\
-1 & \text{if} \ l \ge 1,
\end{cases}
\end{equation*}
which implies
$$S_{0} = 1 - (e^{-\alpha_i} + e^{-2 \alpha_i} + \cdots ) = 1 -
e^{-\alpha_i} \dfrac{1}{1-e^{-\alpha_i}} = \dfrac{1-2
e^{-\alpha_i}}{1-e^{-\alpha_i}}. $$}
\end{example}

\vskip 3mm 

Choose a linear functional $\rho$ on $\h$ such that $\langle h_i,
\rho \rangle =1$ for all $i\in I$. For each $w \in W$, we denote by
$l(w)$ the length of $w$ and set $\epsilon(w)= (-1)^{l(w)}$.
The following proposition gives a character formula for $V(\lambda)$.

\vskip 3mm 

\begin{proposition} \label{prop:character} \cite{BSV2016} \
{\rm Let $V=U_{q}(\g)\, v_{\lambda}$ be a highest weight $U_{q}(\g)$-module with
highest weight $\lambda \in P^{+}$. If $V$ satisfies the conditions 
in Proposition \ref{prop:hw}, then the character of $V$ is given by the
following formula\,:

\begin{equation} \label{eq:character}
\begin{aligned}
\text{ch} V &= \dfrac{\sum_{w \in W} \epsilon (w) e^{w(\lambda +
\rho)-\rho} w (S_{\lambda})}{\prod_{\alpha \in \Delta_{+}} (1-
e^{-\alpha})^{\dim \g_{\alpha}}} \\
&= \dfrac{\sum_{w \in W} \sum_{s \in F_{\lambda}} \epsilon(w)
\epsilon(s) e^{w(\lambda + \rho -s) - \rho}}{\prod_{\alpha \in
\Delta_{+}} (1-e^{-\alpha})^{\dim \g_{\alpha}}}.
\end{aligned}
\end{equation}

\vskip 1mm

In particular, the character of $V(\lambda)$ is given by this formula. 
}
\end{proposition}

\vskip 2mm

\begin{proof}
The proof given in \cite[Theorem 6.1]{BSV2016} depends only on the conditions in
Proposition \ref{prop:hw},  not on the irreducibility of $V$. Hence their argument  works for any
highest weight $U_{q}(\g)$-module with highest weight $\lambda \in P^{+}$
satisfying the conditions in Proposition \ref{prop:hw}.
\end{proof}

\vskip 3mm

\begin{remark} \
{\rm Here, $\g_{\alpha}$ $(\alpha \in \Delta_{+})$ denotes the root space of the Borcherds-Bozec algebra $\g$ associated with the same Borcherds-Cartan datum as $U_{q}(\g)$. }
\end{remark}
 
\vskip 3mm

\begin{corollary} \
{\rm Let $V=U_{q}(\g) v_{\lambda}$ be a highest weight module with highest weight $\lambda \in P^{+}$. If $V$ satisfies the conditions in Proposition \ref{prop:hw}, then $V \cong V(\lambda)$. 
}
\end{corollary}

\begin{proof} \, This is an immediate consequence of Proposition \ref{prop:character}. 
\end{proof}

\vskip 3mm 

\begin{corollary} \label{cor:simple} \hfill

\vskip 1mm 

{\rm
\begin{enumerate}
\item[(a)] Let $V=U_{q}(\g) v_{\lambda}$ be a highest weight module with highest weight $\lambda \in P$. If $V$ is an object in the category ${\mathcal O}_{\text{int}}$, then $\lambda \in P^{+}$ and $V \cong V(\lambda)$.

\vskip 2mm 

\item[(b)] Every simple object in the category ${\mathcal O}_{\text{int}}$ is isomorphic to some $V(\lambda)$ with $\lambda \in P^{+}$.
\end{enumerate}
}
\end{corollary}

\begin{proof} \ (a) Suppose that $V$ is an object in ${\mathcal O}_{\text{int}}$. If $i \in  I^{\text{re}}$, $f_{i}$ is locally nilpotent on $V$ and by the standard $U_{q}(sl_{2})$-theory, we have $\langle h_{i}, \lambda \rangle \ge 0$, \, $f_{i}^{\langle h_{i}, \lambda \rangle +1} v_{\lambda} =0$. 

\vskip 2mm 

If $i \in I^{\text{im}}$, by the condition (iv) in Definition \ref{def:Oint}, we have $\langle h_{i}, \lambda \rangle \ge 0$ and $\langle h_{i}, \lambda \rangle =0$ implies $f_{il} \, v_{\lambda} =0$ for all $(i,l) \in I^{\infty}$. 

\vskip 2mm 

Hence $\lambda \in P^{+}$ and $V$ satisfies the conditions in Proposition \ref{prop:hw}, which proves our claim. 

\vskip 2mm (b)  Let $V$ be an irreducible $U_{q}(\g)$-module in the category ${\mathcal O}_{\text{int}}$. Then $V$ must be a highest weight module because by Definition \ref{def:Oint}, there is at least one maximal vector in $V$ and any maximal vector would generate a highest weight submodule. By (a), $V \cong V(\lambda)$ for some $\lambda \in P^{+}$.   
\end{proof}

\vskip 3mm 

We will now prove that every $U_{q}(\g)$-module in the category ${\mathcal O}_{\text{int}}$ is completely reducible following the outline given in \cite{HK02}  and \cite{JKK05}. 

\vskip 3mm 

Define an  {\it anti}-involution $\varphi: U_{q}(\g) \rightarrow U_{q}(\g)$ by 
$$ e_{il} \longmapsto f_{il}, \quad f_{il} \longmapsto e_{il}, \quad q^{h} \longmapsto q^{-h} \ \ \text{for all} \ (i,l) \in I^{\infty}, \ h \in P^{\vee}.$$

\vskip 1mm

Let $M=\bigoplus_{\mu \in P} M_{\mu}$ be a $U_{q}(\g)$-module in the category ${\mathcal O}_{\text{int}}$ and set 
$$M^{*}:=\bigoplus_{\mu \in P} M_{\mu}^{*}, \ \ \text{where} \ \ M_{\mu}^{*} = \text{Hom}_{\Q(q)}(M_{\mu}, \Q(q)).$$ 
We define a $U_{q}(\g)$-module structure on $M^{*}$ by
$$\langle x \, \psi, m \rangle := \langle \psi, \varphi(x) \, m \rangle \ \ \text{for} \ x \in U_{q}(\g), \, \psi \in M^{*}, \, m \in M.$$ 

\vskip 2mm 

\begin{lemma} \label{lem:Mdual} \ 
{\rm Let $M =\bigoplus_{\mu \in P} M_{\mu}$ be a $U_{q}(\g)$-module in the category ${\mathcal O}_{\text{int}}$.

\begin{enumerate}

\item[(a)] There is a canonical isomorphism $(M^{*})^{*} \cong M$ as $U_{q}(\g)$-modules. 

\item[(b)] $(M^{*})_{\mu} = M_{\mu}^{*}$ for all $\mu \in P$.

\item[(c)] $\text{wt}(M^{*}) = \text{wt}(M)$.

\item[(d)] $M^{*}$ is an object of ${\mathcal O}_{\text{int}}$.  
\end{enumerate}}
\end{lemma}

\vskip 2mm

\begin{proof} \, (a) They are canonically isomorphic as vector spaces. Hence it suffices to verify that the canonical linear map commutes with the $U_{q}(\g)$-module action, which is straightforward. 

\vskip 2mm

(b)  is clear and (c) follows from (b).

\vskip 2mm 

(d) Clearly, $M^{*}$ satisfies the conditions (i) - (iv) in Definition \ref{def:Oint}. Thus we will check the conditions (v) and (vi). 

\vskip 2mm 

Let $i \in I^{\text{im}}$, $\langle h_{i}, \mu \rangle =0$ and $\psi \in M_{\mu}^{*}$. Since $\text{wt}(f_{il}(M_{\mu}^{*}))=\mu - l \alpha_{i}$, we have only to check the condition (v) for the vectors $m \in M_{\mu - l \alpha_{i}}$. Since $\langle h_{i}, \mu - l \alpha_{i} \rangle = - l a_{ii}$, by the definition of $U_{q}(\g)$-module action on $M^{*}$ and the condition (vi) in Definition \ref{def:Oint} for $M$, we have 
$\langle f_{il} \, \psi, m \rangle = \langle \psi, e_{il} \, m \rangle =0,$ which verifies the condition (v) for $M^{*}$. 

\vskip 1mm

Similarly, suppose $i \in I^{\text{im}}$, $\langle h_{i}, \mu \rangle \le - l a_{ii}$ and let $\psi \in M_{\mu}^{*}$. Since $\text{wt}(e_{il}\, \psi) = \mu + l \alpha_{i}$, we take a non-zero vector $m \in M_{\mu + l \alpha_{i}}$. Then we must have $\langle h_{i}, \mu + l \alpha_{i} \rangle \ge 0$. On the other hand, $\langle h_{i}, \mu + l \alpha_{i} \rangle = \langle h_{i}, \mu \rangle + l a_{ii} \le 0$, which implies $\langle h_{i}, \mu + l \alpha_{i} \rangle =0$. Hence by the condition (v) in Definition \ref{def:Oint} for $M$, we have 
$$ \langle e_{il} \, \psi, m \rangle = \langle \psi, f_{il} \, m \rangle =0 \ \ \text{for all} \ l>0,$$ as desired. 
\end{proof}

\vskip 3mm

Let us proceed to prove the complete reducibility of $U_{q}(\g)$-modules in the category ${\mathcal O}_{\text{int}}$. Let $M$ be a $U_{q}(\g)$-module in the category ${\mathcal O}_{\text{int}}$ and let $v_{\lambda}$ be a maximal vector of weight $\lambda$ in $M$. Then the submodule $V$ generated by $v_{\lambda}$ is isomorphic to $V(\lambda)$ and $\lambda \in P^{+}$.  Take a linear functional $\psi_{\lambda} \in M_{\lambda}^{*}$ such that $\langle \psi_{\lambda}, v_{\lambda} \rangle =1$, $\langle \psi_{\lambda}, M_{\mu} \rangle =0$ for all $\mu \neq \lambda$. Then it is easy to  verify that $\psi_{\lambda}$ is a maximal  vector of weight $\lambda$ in $M^{*}$.   Hence the submodule $W:=U_{q}(\g) \, \psi_{\lambda}$ of $M^{*}$ is isomorphic to $V(\lambda)$. 

\vskip 3mm 

The following Lemma is a critical ingredient in proving our main result.  

\vskip 3mm

\begin{lemma} \label{lem:split} \ {\rm Let $M$ be a $U_{q}(\g)$-module in the category ${\mathcal O}_{\text{int}}$ and let $V$ be the submodule of $M$ generated by a maximal vector $v_{\lambda}$ of weight $\lambda$. Then we have 
$$M \cong V \oplus M\big / V.$$
}
\end{lemma}

\begin{proof} \ Consider the short exact sequence 
\begin{equation} \label{eq:split}
0 \longrightarrow V \overset{\iota} \hookrightarrow M \longrightarrow M \big / V \longrightarrow 0.
\end{equation}

\vskip 1mm

We need to prove the sequence \eqref{eq:split} splits. 

\vskip 3mm

Let $W = U_{q}(\g) \, \psi_{\lambda}$ be the submodule of $M^{*}$ described above. Take the dual of the inclusion $W \hookrightarrow M^{*}$ to get a homomorphism $M^{**} \rightarrow W^{*}$. Let  $\eta:M \overset{\sim} \longrightarrow  M^{**} \longrightarrow W^{*}$ be the composition of homomorphisms to obtain
$$\eta \circ \iota: V \hookrightarrow M \overset{\sim} \longrightarrow  M^{**} \longrightarrow W^{*}.$$

\vskip 1mm

Note that the image of the maximal  vector of $V$ is non-zero under the homomorphism $\eta \circ \iota$. Since $W^{*} \cong W \cong V(\lambda)  \cong V$, by Schur's Lemma, $\eta \circ \iota$ is an isomorphism. Hence we get a homomorphism  
$$(\eta \circ \iota)^{-1} \circ \eta : M \overset{\sim} \longrightarrow M^{**} \longrightarrow W^{*} \longrightarrow V.$$

\vskip 1mm 

Clearly, the composition of the homomorphisms 
$$V \overset{\iota} \hookrightarrow M \overset{\eta} \longrightarrow W^{*} \overset{(\eta \circ \iota)^{-1}} \longrightarrow V$$
is the identity and hence the exact sequence  \eqref{eq:split} splits. 
\end{proof}

\vskip 3mm

Now we prove that the category ${\mathcal O}_{\text{int}}$ is semi-simple. 

\vskip 3mm

\begin{theorem} \label{thm:semi-simple} \, {\rm
Every $U_{q}(\g)$-module in the category ${\mathcal O}_{\text{int}}$ is completely reducible.
}
\end{theorem}

\vskip 1mm

\begin{proof} \, Let $M$ be a $U_{q}(\g)$-module in the category ${\mathcal O}_{\text{int}}$. We will prove our claim in two steps. 

\vskip 2mm

{\bf Step 1\,:} \ If $M = U_{q}(\g)\, V$ for some finite-dimensional $U^{\ge 0}$-submodule $V$, then $M$ is completely reducible.  

\vskip 2mm 

We will use induction on the dimension of $V$. If $V=0$, our claim is trivial. If $V \neq 0$, there 
 a maximal  vector $v$ in $V$ with a dominant integral weight $\lambda \in P^{+}$.  Then the submodule $W$ generared by $v$ is isomorphic to $V(\lambda)$ and by Lemma \ref{lem:split}, we have $M \cong W \oplus M \big / W$. Since $M \big / W \cong U_{q}(\g) (V \big / V \cap W)$ and  
$\dim_{\Q(q)} (V \big / V \cap W) < \dim_{\Q(q)} V$, $M\big / W$ is completely reducible. 

\vskip 3mm 

{\bf Step 2\,:} \ For every $v \in M$, set $V(v):= U^{\ge 0} \, v$, which is a finite-dimensional $U^{\ge 0}$-module due to the condition (ii) in Definition \ref{def:Oint}. By Step 1, $U_{q}(\g) \, v =U_{q}(\g)\,V(v)$ is completely reducible. 
Hence $M=\sum_{v \in M} U_{q}(\g) \, v$ is a sum of irreducible $U_{q}(\g)$-submodules. 
Now by a general argument on semi-simplicity \cite[Proposition 3.12]{CR}, a sum of irreducible submodules is a direct sum. 
Hence we conclude $M$ is completely reducible.  
\end{proof}

\vskip 3mm 

As an immediate consequence, we obtain the following corollary.

\vskip 2mm 

\begin{corollary} \label{cor:string} \, {\rm Let $M$ be a $U_{q}(\g)$-module in the category ${\mathcal O}_{\text{int}}$
and let $U_{(i)}$ be the quantum string subalgebra corresponding to $i \in I$. 

\vskip 2mm
\begin{enumerate}

\item[(a)] If $i \in I^{\text{re}}$, $M$ is isomorphic to a direct sum of finite-dimensional irreducible $U_{q}(sl_{2})$-modules. 

%\vskip 2mm

\item[(b)] If $i \in I^{\text{im}}$, $M$ is isomorphic to a direct sum of 1-dimensional trivial modules and infinite-dimensional irreducible highest weight modules over $U_{(i)}$. 
\end{enumerate}}
\end{corollary}

\begin{proof} \, For each $i \in I$, using the same argument in this section, one can verify that $M$ is completely reducible as a $U_{(i)}$-module.  Our assertions follow from the observation in Example \ref{ex:string}.
\end{proof}

\end{document}